\documentclass[draft]{amsart}
\usepackage{amsmath,amscd,amssymb}
\usepackage[arrow, matrix, curve]{xy}
\usepackage{tikz}
\usetikzlibrary{patterns}

\DeclareMathAlphabet{\mathbf}{OML}{cmm}{b}{it}

\newcommand{\FF}{\mathbb F}
\newcommand{\NN}{\mathbb N}
\newcommand{\QQ}{\mathbb Q}
\newcommand{\ZZ}{\mathbb Z}

\newcommand{\cA}{\mathcal A}
\newcommand{\cC}{\mathcal C}
\newcommand{\cL}{\mathcal L}

\newcommand{\cU}{\mathcal U}

\newcommand{\fq}{\mathfrak q}

\newcommand{\bA}{\mathbf A}
\newcommand{\bB}{\mathbf B}
\newcommand{\bC}{\mathbf C}
\newcommand{\bD}{\mathbf D}
\newcommand{\bM}{\mathbf M}
\newcommand{\bN}{\mathbf N}

\newcommand{\ann}{\operatorname{ann}}
\newcommand{\id}{\operatorname{id}}
\newcommand{\Aut}{\operatorname{Aut}}
\newcommand{\cha}{\operatorname{char}}

\newcommand{\Ult}{\operatorname{Ult}}
\newcommand{\Frac}{\operatorname{Frac}}

\newcommand{\Spec}{\operatorname{Spec}}
\newcommand{\Max}{\operatorname{Max}}
\newcommand{\Min}{\operatorname{Min}}
\newcommand{\trdeg}{\operatorname{trdeg}}

\def\Jac{\operatorname{Jac}}
\def\Nil{\operatorname{Nil}}

\newcommand{\abs}[1]{\lvert#1\rvert}

\def\<{\langle}
\def\>{\rangle}

\renewcommand\bar{\overline}
\renewcommand\tilde{\widetilde}

\newtheorem{theorem}{Theorem}[section]
\newtheorem*{theorem*}{Theorem}
\newtheorem{lemma}[theorem]{Lemma}
\newtheorem{corollary}[theorem]{Corollary}
\newtheorem{cor}[theorem]{Corollary}
\newtheorem*{cor*}{Corollary}

\newtheorem{prop}[theorem]{Proposition}

\theoremstyle{definition}

\theoremstyle{remark}

\newtheorem*{remark*}{Remark}
\newtheorem*{remarks*}{Remarks}

\newtheorem*{question}{Question}
\newtheorem{example}[theorem]{Example}
\newtheorem{examples}[theorem]{Examples}

\begin{document}
 
\date{October 2016}

\title{The logical complexity of finitely generated commutative rings}
\author[M.~Aschenbrenner]{Matthias Aschenbrenner} \author[A.~Khelif]{Anatole Kh\'{e}lif} \author[E.~Naziazeno]{Eudes Naziazeno} \author[T.~Scanlon]{Thomas Scanlon}

\thanks{Scanlon's work is partially supported by NSF grant DMS-1363372.}


\address{University of California, Los Angeles \\ Department of Mathematics \\
Box 951555 \\
Los Angeles, CA 90095-1555 \\ USA}
\email{matthias@math.ucla.edu}

\address{\'Equipe de Logique Math\'ematique \\
Universit\'e Paris Diderot  \\
UFR de math\'ematiques case 7012, site Chevaleret \\
75205 Paris Cedex 13\\ France}

\email{khelif@math.univ-paris-diderot.fr}

\address{Departamento de Matem\'atica\\
Universidade Federal de Pernambuco\\
Av. Jornalista Anibal Fernandes, s/n, Cidade Universit\'aria, 50740-560, Recife-PE\\ Brasil}
\email{eudes@dmat.ufpe.br}

\address{University of California, Berkeley \\ Department of Mathematics \\
Evans Hall \\
Berkeley, CA 94720-3840 \\ USA}
\email{scanlon@math.berkeley.edu}

\begin{abstract}
We characterize those finitely generated commutative rings which are (parametrically) bi-interpretable with arithmetic:  a finitely generated commutative ring $A$   is bi-interpretable with $(\NN,{+},{\times})$ if and only if  
the space of non-maximal prime ideals of $A$ is
nonempty and connected in the Zariski topology and
 the nilradical of $A$ has a nontrivial annihilator in $\ZZ$.   Notably, by constructing a nontrivial derivation on a nonstandard model of arithmetic we show that the ring of dual numbers over $\ZZ$ is \emph{not} bi-interpretable with $\NN$.
\end{abstract}

\maketitle

\section*{Introduction}

\noindent
We know since G\"odel that the class of arithmetical sets, that is,   
sets definable  in the semiring $(\NN,{+},{\times})$, is very rich; in particular, the first-order theory of this structure is undecidable.
One expects  other mathematical structures which are connected to arithmetic to share this feature.
For instance, since the subset~$\NN$ of $\ZZ$ is definable in the ring $\ZZ$ of integers (Lagrange's Four Square Theorem), every 
subset of~$\NN^m$ which is definable in arithmetic is definable in $\ZZ$.
The usual presentation of integers as differences of natural numbers
(implemented in any number of ways) shows conversely that $\ZZ$ is {\it interpretable}\/ in $\NN$;
therefore every $\ZZ$-definable subset of~$\ZZ^n$ also corresponds to an $\NN$-definable set.
Thus the semiring $\NN$ is interpretable (in fact, definable) in the ring $\ZZ$, and conversely,
$\ZZ$ is interpretable in $\NN$; that is,~$\NN$ and $\ZZ$ are {\it mutually interpretable.}\/ However, something much stronger holds: 
the structures $\NN$ and $\ZZ$ are {\it bi-interpretable.}\/ 

Bi-interpretability is an equivalence relation on the class of first-order structures 
which
captures what it means for two structures (in possibly different languages) to have  essentially have the same categories
of definable sets and maps. (See \cite{AZ} or \cite[Section~5.4]{Ho}.)  Thus in this sense, the definable sets in structures which are bi-interpretable with arithmetic are just as complex 
as those in $(\NN,{+},{\times})$.
We recall the definition of bi-interpretability and its basic properties in Section~\ref{sec:prelims interpretations} below. For example, we show there that  a structure $\bA$ with underlying set $A$ is bi-interpretable with arithmetic
if and only if there are binary operations 
 $\oplus$ and $\otimes$ on $A$ such that
$(\ZZ,{+},{\times})\cong (A,{\oplus},{\otimes})$, and the structures 
$(A,{\oplus},{\otimes})$ and $\bA=(A,\dots)$ have the same definable sets.
The reader who is not yet familiar
with this notion may simply take this equivalent statement   as the definition of ``$\bA$ is bi-interpretable with~$\NN$.''
Bi-interpretability between general structures is a bit subtle and sensitive, for example, to  whether parameters are allowed. Bi-interpretability with~$\NN$ is   more robust,
but we should note here that even for natural algebraic examples, mutual interpretability with~$\NN$ does not automatically entail bi-interpretability with~$\NN$:
for instance, the Heisenberg group $\operatorname{UT}_3(\ZZ)$ of unitriangular $3\times 3$-matrices with entries in $\ZZ$, although it interprets arithmetic \cite{Malcev}, is not bi-interpretable with it; see \cite[Th\'eor\`eme~6]{Kh} or \cite[Theorem~7.16]{Nies}. See \cite{Las02} for interesting examples of  finitely generated simple groups which {\it are}\/ bi-interpretable with $\NN$.

Returning to the commutative world, the consideration of $\NN$ and $\ZZ$ above leads to a natural question: {\it are \emph{all} infinite finitely generated commutative rings
bi-inter\-pretable with $\NN$?}\/
Indeed, each finitely generated commutative ring is interpretable in~$\NN$ (see Corollary~\ref{cor:fg ring int in ZZ} below), and it is known that conversely each infinite finitely generated commutative ring interprets arithmetic~\cite{No}.
However,  it is fairly easy to see as
a consequence of the Feferman-Vaught Theorem that~$\ZZ\times\ZZ$ is not bi-interpretable
with $\NN$. Perhaps more surprisingly, there are nontrivial derivations on nonstandard models
of arithmetic and it follows, for instance, that the ring~$\ZZ[\epsilon]/(\epsilon^2)$ of dual numbers over $\ZZ$ is not bi-interpretable with~$\NN$. (See Section~\ref{sec:derivations}.)

The main result of this paper is a characterization of the
 finitely generated commutative rings which are
bi-interpretable with $\NN$. To formulate it, we need some notation.
Let $A$ be a commutative ring (with unit). 
As usual, we write~$\Spec(A)$ for the spectrum of $A$, i.e., the set of prime ideals of $A$ equipped with the Zariski topology, and $\Max(A)$ for the subset of $\Spec(A)$ consisting of the maximal ideals of~$A$. We put
$\Spec^\circ(A) := \Spec(A) \setminus\Max(A)$, equipped with the subspace topology. (In the context of a local ring $(A,\mathfrak m)$, the topological space $\Spec^\circ(A)={\Spec(A)\setminus\{\mathfrak m\}}$ is known as the ``punctured spectrum'' of $A$.)

\begin{theorem*}
Suppose the ring $A$ is finitely generated, and let $N$ be the nilradical of~$A$. Then $A$ is bi-interpretable with $\NN$
if and only if $A$ is infinite, $\Spec^\circ(A)$ is connected, and there is some integer $d\geq 1$ with $dN=0$.
\end{theorem*}

\noindent
The proof of the theorem is contained in Sections~\ref{sec:integral domains}--\ref{sec:derivations}, preceded by
two preliminary sections, on algebraic background and on interpretations, respectively.
Let us indicate the strategy of the proof. Clearly if $A$ is bi-interpretable with $\NN$, then necessarily $A$ is
infinite.
Note that the theorem says in particular that if~$A$ is an infinite integral domain, then $A$ is bi-interpretable with~$\NN$. We prove this fact in Section~\ref{sec:integral domains} using techniques of~\cite{Sc} which are unaffected by
the error therein~\cite{ScErr}, as sufficiently many
valuations on the field of fractions of $A$ may be defined via ideal membership conditions in $A$. Combining this fact with Feferman-Vaught-style arguments, in Section~\ref{sec:fiber products} we
then establish the theorem in the case where $A$ is infinite and reduced (that is, $N=0$): $A$ is bi-interpretable with $\NN$ iff $\Spec^\circ(A)$ is connected. To treat the general case, we distinguish two cases according to
whether or not there exists an integer~${d\geq 1}$ with $dN=0$. 
In Section~\ref{sec:finite nilpotent ext}, assuming that there is such a~$d$, we use Witt vectors 
to construct a bi-interpretation between $A$ and its associated reduced ring $A_{\operatorname{red}}=A/N$.
Noting that $A$ is finite if and only if $A_{\operatorname{red}}$ is finite, and $\Spec^\circ(A)$ and  $\Spec^\circ(A_{\operatorname{red}})$ are homeomorphic, this allows us to appeal to the case of a reduced ring $A$. Finally, by constructing suitable automorphisms of an elementary extension of $A$
we prove that if  there is no such integer $d$, then $A$ cannot be bi-interpretable with $\NN$. (Section~\ref{sec:derivations}.)

Structures bi-interpretable with arithmetic are ``self-aware'': they know their own isomorphism type. More
precisely, 
if a finitely generated structure $\bA$ in a finite language $\cL$ is bi-interpretable with~$\NN$,
then $\bA$ is   \emph{quasi-finitely axiomatizable~\textup{(}QFA\textup{)},} that is, there is an $\cL$-sentence $\sigma$  satisfied by $\bA$ such that every finitely generated $\cL$-structure satisfying~$\sigma$ is isomorphic to $\bA$; see Proposition~\ref{prop:QFA} below. (This notion of quasi-finite axiomatizability does not agree with the one commonly used in Zilber's program, for example, in~\cite{AZ}.)
In~\cite{Nies03}, Nies  first considered 
the class of QFA groups, which has been studied extensively since then; see,
for example~\cite{Las01,Las02,Las03,LO,Nies07,Oger,OgerSabbagh}.

In 2004, Sabbagh~\cite[Theorem~7.11]{Nies} gave a direct argument for the quasi-finite axiomatizability of the ring of integers. 
Belegradek~\cite[\S{}7.6]{Nies} then raised the question which finitely generated commutative rings are QFA.
Building on our result that finitely generated integral domains are bi-interpretable with $\NN$, in the last section of this paper we prove:

\begin{cor*}
Each finitely generated commutative ring is QFA.
\end{cor*}

\noindent
This paper had a rather long genesis, which we briefly summarize. 
Around 2005, the second- and fourth-named authors independently realized that bi-interpretability with $\NN$ entails QFA. The fourth-named author was motivated by Pop's 2002 conjecture~\cite{Pop} that finitely generated fields are determined up to isomorphism by their elementary theory. In \cite{Sc} the fourth-named author attempted to establish this conjecture by showing that they are bi-interpretable with~$\NN$; however, later, Pop found a mistake in this argument, and his conjecture remains open~\cite{ScErr}. 
Influenced by  \cite{Sc} and realizing that not all finitely generated commutative rings are bi-interpretable with~$\NN$,
in 2006 the first-named author  became interested in algebraically characterizing those which are.
The corollary above was announced by the second-named author in \cite{Kh}, where a proof based on the main result of \cite{Sc} was suggested. In his Ph.~D.~thesis~\cite{Na}, the third-named author later gave a  proof of this corollary circumventing the flaws of~\cite{Sc}. 

\medskip
\noindent
We conclude this introduction with an open question suggested by our theorems above.
Recall that a group $G$ is said to be {\it metabelian}\/ if its commutator subgroup~$G'=[G,G]$ is abelian.
If $G$ is a  metabelian group, then the abelian group~$G/G'$ can be made into a
module $M$ over the group ring $A=\ZZ[G']$ in a natural way; if moreover $G$ is finitely generated, then 
the commutative ring~$A$ is finitely generated, and so is the $A$-module $M$, hence 
by the above, the two-sorted structure~$(A,M)$ is~QFA. (Lemma~\ref{lem:module QFA}.) However, no infinite abelian group is QFA~\cite[\S{}7.1]{Nies}, and we already mentioned that
the metabelian group $\operatorname{UT}_3(\ZZ)$ is not bi-interpretable with $\NN$, though it is~QFA~\cite[\S{}7.2]{Nies}. The second-named author has shown that every non-abelian free metabelian group is bi-interpretable with~$\NN$~\cite{LogBlog}.
Each non-abelian finitely generated metabelian  group interprets~$\NN$~\cite{No84}.

\begin{question}
Is every non-abelian finitely generated metabelian group QFA?
Which finitely generated metabelian groups are bi-interpretable with $\NN$?
\end{question}

\subsection*{Notations and conventions}
We let $m$, $n$ range over the set $\NN=\{0,1,2,\dots\}$ of natural numbers. In this paper, ``ring'' always means ``commutative ring with unit.'' Rings are always viewed as model-theoretic structures in the language $\{{+},{\times}\}$ of rings.
We occasionally abbreviate ``finitely generated'' by ``f.g.''
The adjective ``definable'' will always mean ``definable, possibly with parameters.''

\section{Preliminaries: Algebra}\label{sec:prelims algebra}

\noindent
In this section we gather some basic definitions and facts of a ring-theoretic nature which are used later.

\subsection{Radicals}
Let $A$ be a ring and $I$ be an ideal of $A$.  We denote by $\Nil(I)$ the {\it nilradical}\/ of $I$, that is, the ideal
$$\Nil(I) := \big\{ a\in A: \exists n\ a^n \in I \big\}$$
of $A$, and we write
$$\Jac(I) := \big\{ a\in A: \forall b\in A\ \exists c\in A\ (1-ab)c\in 1+I \big\}$$
for the {\it Jacobson radical}\/ of $I$. 
It is well-known that $\Nil(I)$ equals
the intersection of all prime ideals of $A$ containing $I$, and $\Jac(I)$ equals the intersection of all maximal ideals of $A$ which contain $I$.
Evidently, $I\subseteq\Nil(I)\subseteq\Jac(I)$. The ideal $I$ is said to be {\it radical}\/ if $\Nil(I)=I$. 
For our purposes it is important to note that although the nilradical is not uniformly definable for all rings, the Jacobson radical is; more precisely, we have: if $\varphi(x)$ is a formula  defining $I$ in  $A$, then the formula
$$\Jac(\varphi)(x) := \forall u \exists v \exists w \big((1-xu)v = 1+w \ \&\ \varphi(w)\big)$$
defines $\Jac(I)$ in $A$.
We denote by~$N(A)$ the nilradical   of the zero ideal of~$A$.
Thus $N(A)=\bigcap_{\mathfrak p\in\Spec A} \mathfrak p$.
One says that~$A$ is \emph{reduced} if $N(A)=0$. The ring $A_{\operatorname{red}}:=A/N(A)$ is reduced, and called
the associated reduced ring of $A$.
We say that $I$  is \emph{nilpotent} if there is some integer $e\geq 1$ such that $I^e=0$. 
The smallest such $e$ is the \emph{nilpotency index} of $I$ (not to be confused with the index $[A:I]$
of $I$ as an additive subgroup of $A$).  If $N(A)$ is f.g., then it is nilpotent.  
 
\begin{lemma}\label{lem:finite ring}
$A$ is finite if and only if it contains a f.g.~nilpotent ideal of finite in\-dex in $A$.
\textup{(}In particular, if $N(A)$ is f.g., then $A$ is finite iff $A_{\operatorname{red}}$ is finite.\textup{)}
\end{lemma}
\begin{proof}
Let $N$ be a f.g.~ideal of $A$ such that $A/N$ is finite, and $e\geq 1$ such that $N^e=0$. We show, by induction on $i=1,\dots,e$, that $A/N^i$ is finite. The case $i=1$ holds by assumption. Suppose now that we have already shown that $A/N^i$ is finite, where $i\in\{1,\dots,e-1\}$. Then $N^i/N^{i+1}$ is an $A/N$-module in a natural way, and f.g.~as such, hence finite. Since $A/N^i\cong (A/N^{i+1})/(N^i/N^{i+1})$, this yields that $A/N^{i+1}$ is also finite.
\end{proof}

\subsection{Jacobson rings}
{\it In this subsection we let $A$ be a ring.}\/ One calls $A$ a {\it Jacobson ring}\/ (also sometimes a {\it Hilbert ring}\/) if every prime ideal of $A$ is an intersection of maximal ideals; that is, if $\Nil(I)=\Jac(I)$ for every ideal $I$ of $A$.
The class of Jacobson rings is  closed under taking homomorphic images: if $A\to B$ is a surjective ring morphism and $A$ is a Jacobson ring, then $B$ is a Jacobson ring. 
Examples for Jacobson rings include all fields and the ring $\ZZ$ of integers, or more generally, every principal ideal domain with infinitely many pairwise non-associated primes.
The main interest in Jacobson rings in commutative algebra and algebraic geometry is their relation with Hilbert's Nullstellensatz, an abstract version of which states that if $A$ is a Jacobson ring, then so is any f.g.~$A$-algebra $B$; in this case, the pullback of any maximal ideal $\frak n$ of $B$ is a maximal ideal $\frak m$ of~$A$, and $B/\frak n$ is a finite extension of the field $A/\frak m$. In particular, every f.g.~ring is a Jacobson ring.

\begin{lemma} \label{Field implies finite}
Suppose $A$ is a field which is f.g.~as a ring. Then $A$ is finite.
\end{lemma}
\begin{proof}
The pullback $\frak m$ of the maximal ideal $\{0\}$ of $A$ is maximal ideal of $\ZZ$, that is, $\frak m=p\ZZ$ for some prime number $p$, and $A$ is a finite extension of the finite field $\ZZ/p\ZZ$, hence finite.
\end{proof}

\begin{cor} \label{cor:finite = 0-dim}
Suppose $A$ is f.g. Then
$A$ is finite if and only if $\Spec^\circ(A)=\emptyset$, that is, every prime ideal of $A$ is maximal.
\end{cor}
\begin{proof}
We may assume that $A$ is nontrivial. A nontrivial ring is called zero-dimensional 
if it has no non-maximal prime ideals.
Every nontrivial finite ring (in fact, every nontrivial artinian ring \cite[Example~2, \S{}5]{Mats}) is zero-dimensional. Conversely, assume that $A$ is zero-dimensional. 
Then $A$ has only finitely many pairwise distinct maximal ideals $\frak m_1,\dots,\frak m_k$, and 
setting $N:=N(A)$, we have $N = \frak m_1\cap\cdots\cap\frak m_k$.   Each of the fields $A/\frak m_i$ is f.g.~as a ring, hence finite, by Lemma~\ref{Field implies finite}. By the Chinese Remainder Theorem, $A/N\cong (A/\frak m_1)\times\cdots\times (A/\frak m_k)$, thus $A/N$ is finite. Hence by Lemma~\ref{lem:finite ring}, $A$ is finite.
\end{proof}

\noindent
Given an element $a$ of a ring, we say that $a$ has infinite multiplicative order if $a^m\neq a^n$ for all $m\neq n$.

\begin{cor}\label{cor:inf mult order}
Every infinite f.g.~ring contains an element of infinite multiplicative order.
\end{cor}
\begin{proof}
Let $A$ be f.g.~and infinite, and let $\mathfrak p$ be a non-maximal prime ideal of $A$, according to the previous corollary. Take $a\in A\setminus\mathfrak p$ such that $1\notin (a,\mathfrak p)$. Then $a$ has infinite multiplicative order.
\end{proof}

\noindent
It is a classical fact that if $A$ is noetherian of (Krull) dimension at most $n$, then every radical ideal of $A$ is the nilradical of an ideal generated by~${n+1}$ elements.
(This is due to Kronecker~\cite{Kronecker} in the case where $A$ is a polynomial ring over a field, and to
van der Waerden in general; see \cite{EE}.) 
 
\begin{lemma}\label{lem:def prime ideals}
There exist formulas 
$$\pi_n(y_1,\dots,y_{n+1}),\ \mu_n(y_1,\dots,y_{n+1}),\ \Pi_n(x,y_1,\dots,y_{n+1})$$ with the following property: if $A$ is a noetherian Jacobson ring of dimension at most $n$, then
\begin{align*}
\Spec A 	&= \big\{ \Pi_n(A,a): a\in A^{n+1},\ A\models\pi_n(a) \big\} \\
\Max A 	&= \big\{ \Pi_n(A,a): a\in A^{n+1},\ A\models\mu_n(a) \big\}.
\end{align*}
\end{lemma}
\begin{proof}
For every $n$ let
\begin{align*}
\gamma_n(x,y_1,\dots,y_n)	&:= \exists z_1 \cdots \exists z_n (x=y_1z_1+\cdots+y_nz_r), \\
\Jac_n(x,y_1,\dots,y_n)		&:= \Jac(\gamma_{n}).
\end{align*}
Then for every $n$-tuple $a=(a_1,\dots,a_n)$ of elements of  $A$, the formula $\gamma_n(x,a)$ defines the ideal of $A$ generated by $a_1,\dots,a_n$, and $\Jac_n(x,a)$ 
defines its Jacobson radical. Writing $y$ for $(y_1,\dots,y_{n+1})$, the formulas
\begin{align*}
\pi_n(y) 	&:=  \forall v \forall w \big( \Jac_{n+1}(v\cdot w,y) \rightarrow (\Jac_{n+1}(v,y) \vee \Jac_{n+1}(w,y) \big), \\
\mu_n(y)		&:= \forall v \exists w \big( \Jac_{n+1}(v,y) \vee \Jac_{n+1}(1-vw,y) \big), \\
\Pi_n(x,y) 	&:= \Jac_{n+1}(x,y)
\end{align*}
have the required property, by Kronecker's Theorem.
\end{proof}

\begin{remarks*} 
\ 
 
\begin{enumerate}
\item The previous lemma holds if the noetherianity hypothesis is dropped and  $\Spec A$ and $\Max A$ are replaced with the set of f.g.~prime ideals of $A$ and the set of f.g.~maximal ideals of $A$, respectively, by a non-noetherian analogue of Kronecker's Theorem due to Heitmann \cite[Corollary~2.4,~(ii) and Remark~(i) on p.~168]{Hei}.
\item Let $\pi_n$, $\mu_n$, $\Pi_n$ be as in Lemma~\ref{lem:def prime ideals}, and set
$\pi_n^\circ := \pi_n\wedge\neg\mu_n$. Then for every noetherian Jacobson ring $A$ of  dimension at most $n$ we have
$$\Spec^\circ A 	= \big\{ \Pi_n(A,a): a\in A^{n+1},\ A\models\pi_n^\circ(a) \big\}.$$
Hence for every such ring $A$,
we have $A\models \forall y_1\cdots\forall y_{n+1}\neg\pi_n^\circ$ iff $\dim A<1$.
(Using the inductive characterization of Krull dimension from \cite{CLR}, one can actually construct, for each $n$,
a sentence $\dim_{<n}$ such that for all Jacobson rings $A$, we have
$A\models \dim_{<n}$ iff $\dim A<n$.)
\end{enumerate}
\end{remarks*}

\subsection{Subrings of a localization}
In this subsection we let $R$ be a ring and $D$ be a subring of $R$.

\begin{prop}\label{prop:fg criterion}
Suppose that $D$ is a Dedekind domain and  
$D[c^{-1}]=R[c^{-1}]$ for some $c\in D\setminus\{0\}$. Then $R$ is a f.g.~$D$-algebra.
\end{prop}

\noindent
This can be deduced from \cite[Theorem~2.20]{Onoda}, but we give a direct proof based on
a simple lemma from this paper:

\begin{lemma}[Onoda \cite{Onoda}]\label{lem:ideal}
Suppose $R$ is an integral domain. Then the set of $c\in R$ such that $c=0$ or
$c\neq 0$ and $R[c^{-1}]$ is a f.g.~$D$-algebra is an ideal of $R$.
\end{lemma}
\begin{proof}
Since this set clearly is closed under multiplication by elements of $R$, we only need
to check that it is closed under addition. Let $a_1,a_2\in R\setminus\{0\}$ be such that $a_1+a_2\neq 0$ and the $D$-algebra
$R[a_i^{-1}]$ is f.g., for $i=1,2$. So we can take a f.g.~$D$-algebra $B\subseteq R$ such that $a_i\in B$ and $B[a_i^{-1}]\supseteq R[a_i^{-1}]$, for $i=1,2$. Given $x\in R$, take $n\geq 1$ such that $a_i^n x\in B$;
then $(a_1+a_2)^{2n-1}x\in B$, so $x\in B[(a_1+a_2)^{-1}]$. Thus $R[(a_1+a_2)^{-1}]=B[(a_1+a_2)^{-1}]$ is a f.g.~$D$-algebra.
\end{proof}

\begin{proof}[Proof of Proposition~\ref{prop:fg criterion}]
Let $c$ be as in the statement of the proposition, and first let $S$ be a multiplicative subset of $D$ with $R[S^{-1}]=D[S^{-1}]$;
we claim that then there is some $s\in S$ such that $R[s^{-1}]=D[s^{-1}]$. To see this note that
for each $Q\in\Spec(D)$ with $D_Q\not\supseteq R$ we have $c\in Q$, since otherwise
$D_Q \supseteq D[c^{-1}]=R[c^{-1}]\supseteq R$, and for a similar reason we have $Q\cap S\neq\emptyset$.
Let $Q_1,\dots,Q_m$ be the prime ideals $Q$ of $D$  with $D_Q\not\supseteq R$. For $i=1,\dots,m$
pick some $s_i\in Q_i\cap S$ and set $s:=s_1\cdots s_m\in S\cap Q_1\cap\cdots\cap Q_m$. Then we have $D_Q\supseteq R$ for each $Q\in\Spec(D)$ with $s\notin Q$, and hence $D[s^{-1}]=\bigcap_{s\notin Q} D_Q\supseteq R$. Therefore $D[s^{-1}]=R[s^{-1}]$.

Let now $I$ be the ideal of $R$ defined in Lemma~\ref{lem:ideal}; we need to show that $1\in I$. 
Towards a contradiction assume that we have some prime ideal $P$ of $R$ which contains $I$.
Put $Q:=D\cap P\in\Spec(D)$ and $S:=D\setminus Q$. Then $D_Q=R_P$ (there is no proper intermediate ring between a DVR and its fraction field). In fact, we have $D[S^{-1}]=D_Q=R[S^{-1}]$, so by the above 
there is some $s\in S$ with $R[s^{-1}]=D[s^{-1}]$. Hence $s\in I\setminus P$, a contradiction.
\end{proof}

\begin{remark*}
We don't know whether the conclusion of Proposition~\ref{prop:fg criterion} can be strengthened
to: $R=D[r^{-1}]$ for some $r\in R\setminus\{0\}$.
\end{remark*}

\noindent
For a proof of the next lemma see, e.g., \cite[Proposition~7.8]{AM}.

\begin{lemma}[Artin-Tate \cite{AT}] \label{lem:Artin-Tate}
Suppose $D$ is noetherian  and $R$ is contained in a f.g.~$D$-algebra which is integral over $R$.
Then the $D$-algebra $R$ is also f.g.
\end{lemma}

\noindent
The following fact is used in Section~\ref{sec:integral domains}.

\begin{cor}\label{cor:fg criterion}
Suppose $D$ is a one-dimensional noetherian integral domain whose integral closure $\tilde{D}$ in the fraction field $K$ of $D$ is a f.g.~$D$-module. If $D[c^{-1}]=R[c^{-1}]$ for some $c\in D\setminus\{0\}$, 
then $R$ is a f.g.~$D$-algebra.
\end{cor}

\begin{proof} 
Let $\tilde R$ be the integral closure of  $R$ in $K$. 
Suppose $c\in D\setminus\{0\}$ satisfies $D[c^{-1}]=R[c^{-1}]$. Then $\tilde D$ is a Dedekind domain and $\tilde{D}[c^{-1}]=\tilde{R}[c^{-1}]$. By Proposition~\ref{prop:fg criterion}, $\tilde{R}$ is a f.g.~$\tilde{D}$-algebra, and hence  also a f.g.~$D$-algebra. Lemma~\ref{lem:Artin-Tate} implies that $R$ is a f.g.~$D$-algebra.
\end{proof}


\subsection{Annihilators}
Let $A$ be a ring.
Given an $A$-module $M$ we denote by
$$\ann_A(M) := \{ a\in A:  aM=0 \}$$
the annihilator of $M$ (an ideal of $A$), and if $x$ is an element of $M$ we also write $\ann_A(x)$ for the annihilator of the submodule $Ax$ of $M$, called the annihilator of~$x$.
The annihilator $\ann_\ZZ(A)$ of $A$ viewed as a $\ZZ$-module is either the zero ideal, in which case we say that the characteristic of $A$ is $0$, or contains a smallest positive integer,  called the {\it characteristic}\/ of $A$. (Notation: $\operatorname{char}(A)$.)

\medskip
\noindent
In the following we
let $N:=N(A)$. We also set $A_\QQ:=A\otimes_\ZZ\QQ$, with natural morphism 
$$a\mapsto \iota(a):=a\otimes 1\colon A\to A_\QQ.$$ 
Its kernel is the torsion subgroup
$$A_{\operatorname{tor}} := \big\{ a\in A: \ann_\ZZ(a)\neq 0 \big\}$$
of the additive group of $A$.  
Suppose that the ideal~$A_{\operatorname{tor}}$ of $A$ is
finitely generated. Then 
there is some integer $e\geq 1$  such that $eA_{\operatorname{tor}}=0$;
the smallest such $e$ is called the  {\it exponent}\/ of $A_{\operatorname{tor}}$.
One checks easily that  then
\begin{align*}
\iota^{-1} N(A_\QQ) &= (N:e) := \{ a\in A: ea\in N\}, \\
\iota^{-1} \ann_{A_\QQ} \big(\iota(a)\big) & = \ann_A(ea)\quad\text{for each $a\in A$.}
\end{align*}
The following lemma on the existence of nilpotent elements with prime annihilators is used in Section~\ref{sec:derivations}. (Note that if $\epsilon$ is as in the conclusion of the lemma, then 
$\epsilon^2=0$ and
$A/\ann_A(\epsilon)$ is an integral domain of characteristic zero.)

\begin{lemma}
\label{annprime}
Suppose that $A$ is noetherian and $\ann_\ZZ(N) = 0$. Then there is some $\epsilon \in N$ with $\ann_A(\epsilon)$ prime and $\ann_\ZZ(\epsilon) = 0$.
\end{lemma}
\begin{proof}
Note that  the hypothesis $\operatorname{ann}_\ZZ(N) = 0$ implies not only that~$N$ is nonzero, but also that some nonzero element of $N$ remains nonzero under~$\iota$; in particular, $N(A_\QQ)\neq 0$.
Let~$\cA$ be the set of annihilators of nonzero elements of $N(A_\QQ)$.   Then~$\cA \neq \varnothing$, and as $A_\QQ$ is noetherian, we may find a maximal element $P \in \cA$.  Scaling if need be, we may assume that~$P= \ann_{A_\QQ}(\iota(a))$ where $a \in \iota^{-1} N(A_\QQ)=(N:e)$, $e=\text{exponent of $A_{\operatorname{tor}}$}$.  The ideal~$P$ is prime, as if $xy \iota(a) = 0$ while neither $x \iota(a) = 0$ nor $y \iota(a) = 0$, then $P\subseteq  \ann_{A_\QQ}(y\iota(a))$ with
$x \in \ann_{A_\QQ}(y\iota(a)) \setminus P$, contradicting maximality.  Thus, $\ann_A(ea) = \iota^{-1} P$ is prime and $\ann_\ZZ(ea) = 0$, so $\epsilon:=ea$ does the job.
\end{proof}

\subsection{A bijectivity criterion} 

In the proof of Proposition~\ref{prop:QFA fg ring} we apply the following criterion:

\begin{lemma}\label{lem:criterion for bijection}
Let $\phi\colon A\to B$ be a morphism of additively written abelian groups. Let $N$ be a subgroup of $A$.
Suppose that the restriction of $\phi$ to $N$ is injective, and the morphism $\overline{\phi}\colon A/N\to B/\phi(N)$ induced by $\phi$ is bijective.
Then $\phi$ is bijective.
\end{lemma}
\begin{proof}
Let $a\in A$, $a\neq 0$. If $a\in N$, then $\phi(a)\neq 0$, since the restriction of $\phi$ to $N$ is injective. Suppose $a\notin N$.   Then $\phi(a)\notin \phi(N)$ since $\overline{\phi}$ is injective; in particular, $\phi(a)\neq 0$. Hence $\phi$ is injective.
To prove that $\phi$ is surjective, let $b\in B$.
Since $\overline{\phi}$ is onto, there is some $a\in A$ such that $b-\phi(a)\in\phi(N)$, so $b\in\phi(A)$ as required.
\end{proof}

\section{Preliminaries: Interpretations}\label{sec:prelims interpretations}

\noindent
In this section we recall the notion of interpretation, and record a few consequences (some of which may be well-known) of bi-interpretability with $\NN$. We begin by discussing definability in quotients of definable equivalence relations.
Throughout  this section, we let  $\bA=(A,\dots)$ be a structure in some language $\cL=\cL_\bA$ and $\bB=(B,\dots)$ be a structure in some language $\cL_\bB$.

\subsection{Definability in quotients}
Let $E$ be a definable equivalence relation on a definable set $S\subseteq A^m$, with
 natural surjection $\pi_E\colon S\to S/E$. Note that for $X\subseteq S$ we have $X=\pi_E^{-1}(\pi_E(X))$  
 iff $X$ is {\it $E$-invariant,}\/ that is, for all $(a,b)\in E$
 we have~${a\in X}$ iff $b\in X$.
 A subset of $S/E$ is said to be {\it definable} in $\bA$ if its preimage under $\pi_E$ is definable  in $\bA$; equivalently, if it is the image of some definable subset of $S$ under~$\pi_E$. A map $S/E\to S'/E'$, where $E'$ is a definable equivalence relation on some definable set $S'$ in $\bA$, is said to be definable in $\bA$  if
its graph, construed as a subset of 
$(S/E)\times (S'/E')$, is definable.
Here and below, given an equivalence relation $E$ on a set $S$ and an equivalence relation~$E'$ on~$S'$, we identify~${(S/E)\times (S/E')}$ in the natural way with $(S\times S')/(E\times E')$,
where~$E\times E'$ is the equivalence relation on~$S\times S'$ given by 
$$(a,a')\ (E\times E')\ (b,b') \quad:\Longleftrightarrow\quad \text{$aEb$  and $a'E'b'$} \qquad (a,b\in S, a',b'\in S).$$

\subsection{Interpretations}\label{sec:interpretations}
A surjective map $f\colon M\to B$, where $M\subseteq A^m$ (for some~$m$) is an
{\it interpretation of $\bB$ in $\bA$}\/ (notation: $f\colon \bA\leadsto\bB$) if for every set $S\subseteq B^n$ which is definable in $\bB$,
the preimage $f^{-1}(S)$ of $S$ under the map $$(a_1,\dots,a_n)\mapsto \big(f(a_1),\dots,f(a_n)\big)\colon M^n\to B^n,$$ which we also denote by $f$, is a definable subset of $M^n\subseteq (A^m)^n=A^{mn}$.
It is easy to verify that a surjective map $f\colon M\to B$ ($M\subseteq A^m$) is an interpretation of~$\bB$ in~$\bA$ iff
the kernel
$$\ker f:=\big\{(a,b)\in M\times M: f(a)=f(b)\big\}$$ 
of $f$, as well as the preimages of 
the interpretations (in $\bB$) of each relation symbol and
the graphs of the interpretations of each function symbol  from~$\cL_\bB$, are
definable in $\bA$. 
If the parameters in the formula defining $\ker f$ and in
the formulas defining the preimages of the interpretations of the symbols of $\cL_\bB$ in $\bA$ can be
chosen to come from some set $X\subseteq A$, we say that $f$ is an $X$-interpretation of $\bB$ in $\bA$, or an interpretation of $\bB$ in $\bA$ over $X$. An interpretation $\bA\leadsto\bA$ is called a {\it self-interpretation}\/ of $\bA$. (A trivial example is the identity interpretation $\id_\bA\colon A\to A$.)

\medskip
\noindent
We say that {\it $\bB$ is interpretable in $\bA$}\/ if there exists an interpretation of $\bB$ in $\bA$. Given such an interpretation $f\colon M\to B$ of $\bB$ in $\bA$, we write $\bar{M}:=M/\ker f$ for  the set of equivalence classes of the equivalence relation $\ker f$, and
$\bar{f}$ for the bijective map $\bar{M}\to B$ induced by $f$.  
Then $\bar{M}$ is the universe of a unique $\cL_\bB$-structure  $f^*(\bB)$  such that $\bar{f}$
becomes an isomorphism  $f^*(\bB)\to \bB$. We call the $\cL_\bB$-structure $f^*(\bB)$ the {\it copy of $\bB$ interpreted in $\bA$ via the interpretation $f$.}\/

\medskip
\noindent
The composition of two interpretations $f\colon\bA\leadsto\bB$ and $g\colon\bB\leadsto\bC$ is the interpretation $g\circ f\colon\bA\leadsto\bC$ defined in the natural way: if $f\colon M\to B$   
and $g\colon N\to C$, then
$g\circ f\colon f^{-1}(N)\to C$ is an interpretation of $\bC$ in $\bA$.  In this case, the restriction of~$f$ to a map $f^{-1}(N)\to N$ induces an isomorphism $(g\circ f)^*(\bC)\to g^*(\bC)$ between the copy  $(g\circ f)^*(\bC)=f^{-1}(N)/\ker(g\circ f)$ of $\bC$ interpreted in $\bA$ via $g\circ f$ and the copy $g^*(\bC)=N/\ker g$ of $\bC$ interpreted in $\bB$ via $g$
which makes the diagram
$$\xymatrix{(g\circ f)^*(\bC)\ar[rr] \ar[dr]_{\overline{g\circ f}} & & g^*(\bC)  \ar[dl]^{\overline{g}} \\
& \bC &
}$$
commute.
One verifies easily that the composition of interpretations makes the class of all first-order structures into the objects of a category whose morphisms are the interpretations. 


\medskip
\noindent
Suppose $\bB$ is interpretable in $\bA$  via an $\emptyset$-interpretation $f\colon M\to B$. Then every automorphism $\sigma$ of $\bA$  induces a permutation of $M$ and of $\ker f$, and there is a unique permutation $\overline{\sigma}$ of $B$ such that $\overline{\sigma}\circ f=f\circ \sigma$; this permutation $\overline{\sigma}$ is an automorphism of~$\bB$. The resulting map $\sigma\mapsto\overline{\sigma}\colon \Aut(\bA)\to\Aut(\bB)$ is a continuous group morphism \cite[Theorem~5.3.5]{Ho}, denoted by $\Aut(f)$. We therefore have a covariant functor $\Aut$ from the category of structures and $\emptyset$-interpretations to the category of topological groups and continuous morphisms between them.
(Here the topology on automorphism groups is that described in \cite[Section~4.1]{Ho}.)

\medskip
\noindent
If $\bB$ and $\bB'$ are structures which are interpretable in~$\bA$, then their direct product~$\bB\times\bB'$ is also interpretable in $\bA$; in fact, if $f\colon M\to B$ ($M\subseteq A^m$) is an interpretation $\bA\leadsto\bB$, and $f'\colon M'\to B'$ ($M'\subseteq A^{m'}$) is an interpretation $\bA\leadsto\bB'$, then $f\times f'\colon M\times M'\to B\times B'$ is an interpretation $\bA\leadsto\bB\times \bB'$.

\medskip
\noindent
The concept of interpretation allows for an obvious uniform variant: Let 
$\mathfrak A$ be a class of $\mathcal L$-structures and
 $\mathfrak B$ be a class of structures in a language $\mathcal L'$,  for simplicity of exposition assumed to be relational. A {\it uniform interpretation}\/ of 
$\mathfrak B$ in $\mathfrak A$ is given by the following data: 
\begin{enumerate}
\item $\mathcal L$-formulas $\sigma(z)$, $\mu(x;z)$, and $\varepsilon(x,x';z)$; and
\item for each $n$-ary relation symbol $R$ of $\mathcal L'$ an $\mathcal L$-formula $\rho_R(y_R;z)$.
\end{enumerate}
Here $x$, $x'$ are $m$-tuples of  variables (for some $m$), $y_R$ as in (2) is an $mn$-tuple of variables, and
$z$ is a $p$-tuple of variables (for some $p$). All variables in these tuples are assumed to be distinct.
For $\bA\in\mathfrak A$ set $S^{\bA}:=\big\{s\in A^p:\bA\models\sigma(s)\big\}$.
We require that
\begin{itemize}
\item[(U1)]
 for each $\bA\in\mathfrak A$ and $s\in S^{\bA}$, 
the set $M_s:=\big\{a\in A^m:\bA\models\mu(a;s)\big\}$ is nonempty,
$\varepsilon(x,x';s)$ defines
an equivalence relation $E_s$ on $M_s$, and for each $R\in\mathcal L'$, the set $R_s$ defined by 
$\rho(y_R;s)$ in $\bA$ is $E_s$-invariant.
\end{itemize} 
Letting $\pi_s\colon M_s\to M_s/E_s$ be the natural surjection,
the quotient $M_s/E_s$ then becomes the underlying set of an $\mathcal L'$-structure $\bB_s$ interpreted in $\bA$ by
$\pi_s$. We also require  that 
\begin{itemize}
\item[(U2)]
$\mathbf B_s\in\mathfrak B$ for each 
$\bA\in\mathfrak A$, $s\in S^{\bA}$, and
for each $B\in\mathfrak B$ there are some $\bA\in\mathfrak A$, $s\in S^{\bA}$ such that $\bB\cong\bB_s$.
\end{itemize}
We say that $\mathfrak B$ is {\it uniformly interpretable}\/ in~$\mathfrak A$ if there exists a uniform interpretation of $\mathfrak B$ in $\mathfrak A$. Clearly the relation of uniform interpretability is transitive. 
If $\mathfrak B=\{\bB\}$ is a singleton, we also say that $\bB$ is uniformly interpretable in $\mathfrak A$; similarly if $\mathfrak A$ is a singleton.

\subsection{Homotopy and bi-interpretations}
Following \cite{AZ}, we say that interpretations $f\colon M\to B$ and $f'\colon M'\to B$ of $\bB$ in~$\bA$ are {\it homotopic}\/ (in symbols: $f\simeq f'$) if the pullback
$$[{f=f'}]:=\big\{(x,x')\in M\times M' : f(x)=f'(x')\big\}$$
of $f$ and $f'$
is definable in $\bA$; equivalently, if there exists
an isomorphism $$\alpha\colon f^*(\bB)\to (f')^*(\bB)$$ which is definable in $\bA$ such that $\overline{f'}\circ\alpha=\overline{f}$. So for example
if $f$ is a self-interpretation of $\bA$, then $f\simeq\id_A$ if and only if the isomorphism $\overline{f}\colon f^*(\bA)\to\bA$  is definable in $\bA$.
Homotopy is an equivalence relation on the collection of interpretations of $\bB$ in $\bA$.
Given $X\subseteq A$, we say that interpretations $f\colon\bA\leadsto\bB$ and $f'\colon\bA\leadsto\bB$ are $X$-homotopic if $[{f=f'}]$ is $X$-definable.
It is easy to verify that if the $\emptyset$-interpretations
$f,f'\colon\bA\leadsto\bB$ are $\emptyset$-homotopic, then $\Aut(f)=\Aut(f')$.

\begin{lemma} \label{lem:properties of homotopy}
Let  $f,f'\colon\bA\leadsto\bB$ and  $g,g'\colon\bB\leadsto\bC$. Then
\begin{enumerate}
\item $g\simeq g'\Rightarrow g\circ f\simeq g'\circ f$; and
\item if $g$ is injective, then $f\simeq f'\Rightarrow g\circ f\simeq g\circ f'$.
\end{enumerate}
\end{lemma}
\begin{proof}
For (1), note that if $[g=g']$ is definable in $\bB$,
then $[{g\circ f=g'\circ f}]=f^{-1}\big([{g=g'}]\big)$ is definable in $\bA$, and for (2) note that if $g$ is injective
then we have $[{g\circ f = g\circ f'}] = [{f=f'}]$.
\end{proof}

\noindent
Let $f\colon\bA\leadsto\bB$ and $g\colon\bB\leadsto\bA$. 
One says that the pair $(f,g)$ is a {\it bi-interpretation}\/ between $\bA$ and $\bB$ if $g\circ f\simeq \id_\bA$ and $f\circ g\simeq \id_\bB$;
that is, if the isomorphism $\overline{g\circ f}\colon (g\circ f)^*(\bA)\to\bA$ is definable in $\bA$, and
the isomorphism $\overline{f\circ g}\colon (f\circ g)^*(\bB)\to\bB$ is definable in $\bB$. (See Figure~\ref{fig:bi-int}.)
The relation of bi-interpretability is easily seen to be an equivalence relation on the class of first-order structures.
A bi-interpretation~$(f,g)$ between $\bA$ and $\bB$ is an {\it $\emptyset$-bi-interpretation}\/ if $f$,~$g$ are $\emptyset$-interpretations and $g\circ f$ and $f\circ g$ are $\emptyset$-homotopic to the respective identity interpretations.
If $(f,g)$ is such an $\emptyset$-bi-interpretation between $\bA$ and $\bB$, then~$\Aut(f)$ is
a continuous isomorphism $\Aut(\bA)\to\Aut(\bB)$ with inverse $\Aut(g)$.

\begin{figure}
\begin{tikzpicture}[scale=0.75]

\draw (0.25,0) rectangle (4.25,3);
\draw (2.5,1.5) ellipse (1.45 and 1.2);
\draw[pattern=north west lines] (1.5,1) rectangle (3,17/8);
\draw (4.15,1.5) node[fill=white] {\footnotesize $(g\circ f)^*(\bA)$};

\draw (7.5,1.5) ellipse (1.45*1.45 and 1.45*1.2);
\draw[pattern=north west lines] (6,0.75) rectangle (8.5,2.5);

\draw[pattern=north west lines] (11,0) rectangle (15,3);

\draw[densely dotted] (11,0) -- (6,0.75);
\draw[densely dotted] (15,0) -- (8.5,0.75);
\draw[densely dotted] (11,3) -- (6,2.5);
\draw[densely dotted] (15,3) -- (8.5,2.5);

\draw[dashed] (7.5,1.5+1.45*1.2) -- (2.5,1.5+1.2);
\draw[dashed] (7.5,1.5-1.45*1.2) -- (2.5,1.5-1.2);

\draw (13,-1) node {$\bA$};
\draw (7.5,-1) node {$\bB$};
\draw (9.3,1.6) node[fill=white] {$g^*(\bA)$};

\draw (2.25,-1) node {$\bA$};

\end{tikzpicture}
\caption{Composition $g\circ f$ of  $f\colon\bA\leadsto\bB$ and $g\colon \bB\leadsto\bA$.}\label{fig:bi-int}
\end{figure}

\begin{lemma}\label{lem:equiv of categories}
Let $(f,g)$ be a bi-interpretation between $\bA$ and $\bB$. Then for every subset $S$ of  $B^k$ \textup{(}$k\geq 1$\textup{)} we have
$$\text{$S$ is definable in $\bB$} \quad\Longleftrightarrow\quad \text{$f^{-1}(S)$ is definable in $\bA$.}$$
\end{lemma}
\begin{proof}
The forward direction follows from the definition of ``$f$ is an interpretation of $\bB$ in~$\bA$.'' For the converse,
suppose $f^{-1}(S)$ is definable in $\bA$; then the set $S':=(f\circ g)^{-1}(S)=g^{-1}\big(f^{-1}(S)\big)$ is definable in $\bB$ (since $g$ is an interpretation of~$\bA$ in $\bB$). For $y\in B^k$ we have $y\in S$ iff $(f\circ g)(x)=y$ for some $x\in S'$.
Therefore, since $[f\circ g=\id_B]$ and $S'$ are definable in $\bB$, so is $S$.
\end{proof}

\noindent
The previous lemma may be refined to show that a bi-interpretation between~$\bA$ and~$\bB$ in a natural way gives rise to an equivalence of categories between the category of definable sets and maps in $\bA$ and the category of definable sets and maps in $\bB$. (See \cite{MR}.)

\begin{cor}\label{cor:equiv of categories}
Let $(f,g)$ be as in Lemma~\ref{lem:equiv of categories}, and $f',f''\colon\bA\leadsto\bB$.
If $f'\circ g\simeq f''\circ g$, then $f'\simeq f''$.
\end{cor}
\begin{proof}
Note that $g^{-1}\big([{f'=f''}]\big)=[{f'\circ g=f''\circ g}]$ and use Lemma~\ref{lem:equiv of categories}.
\end{proof}

\subsection{Weak homotopy and weak bi-interpretations}
The notion of bi-inter\-pretability allows for a number of subtle variations, one of which (close to the notion of bi-interpretability used in \cite[Chapter~5]{Ho}) we introduce in this subsection.
Given two interpretations $f\colon\bA\leadsto\bB$ and $f'\colon\bA\leadsto\bB'$ of (possibly different) $\cL_\bB$-structures in $\bA$, we say that $f$ and $f'$ are {\it weakly homotopic}\/ if there is an isomorphism $f^*(\bB)\to(f')^*(\bB')$ which is definable in $\bA$; notation: $f\sim f'$. 
Clearly~$\sim$ is an equivalence relation on the class of interpretations of $\cL_\bB$-structures in $\bA$, and ``homotopic'' implies ``weakly homotopic.'' (Note that $f\simeq f'$ only makes sense if $\bB=\bB'$, whereas  $f\sim f'$ merely implies $\bB\cong\bB'$.) The following is easy to verify, and is a partial generalization of the fact that $f\simeq f'$ implies $\Aut(f)=\Aut(f')$:

\begin{lemma}
Let $f\colon\bA\leadsto\bB$ and $f'\colon\bA\leadsto\bB'$, and let $\beta\colon f^*(\bB)\to (f')^*(\bB')$ be an isomorphism, definable in $\bA$.
Put $$\gamma:=\bar{f'}\circ\beta\circ \bar{f}^{-1}\colon \bB\overset{\cong}{\longrightarrow}\bB'.$$ 
Then $\Aut(f)=\gamma\,\Aut(f')\,\gamma^{-1}$.
\end{lemma}

\noindent
We say that a pair $(f,g)$, where $f\colon\bA\leadsto\bB$ and $g\colon\bB\leadsto\bA$, is a \emph{weak bi-in\-ter\-pre\-ta\-tion} between $\bA$ and $\bB$ if $g\circ f\sim\id_\bA$ and $f\circ g\sim\id_\bB$.  The equivalence relation on the class of first-order structures given by  bi-interpretability  is finer than that of weak bi-interpretability, and in general, might be strictly finer. In Section~\ref{sec:bi-int with ZZ} below we see, however, that as far as bi-interpretability with~$\NN$ is concerned, there is no difference between the two notions.

\subsection{Injective interpretations}
An {\it injective interpretation}\/ of $\bB$ in $\bA$ is an interpretation $f\colon\bA\leadsto\bB$ where $f\colon M\to B$ ($M\subseteq A^m$) is injective (and hence bijective). (See \cite[Section~5.4~(a)]{Ho}.)
We also say that the structure $\bB$ is {\it injectively interpretable in $\bA$} if $\bB$ admits an injective interpretation in $\bA$. 

\medskip
\noindent
An important special case of injective interpretations is furnished by relativized reducts. Recall (cf.~\cite[Section~5.1]{Ho}) that $\bB$ is said to be a {\it relativized reduct}\/ of~$\bA$ if the universe $B$ of $\bB$ is a subset of $A^m$, for some $m$,  definable in $\bA$, and the interpretations of the function and relation symbols of $\cL_\bB$ in $\bB$ are definable in $\bA$. In this case, $\bB$ is injectively interpretable in $\bA$, with the interpretation given by the identity map on $B$.  

\begin{example}
The semiring $(\NN,{+},{\times})$ is a relativized reduct of the ring $(\ZZ,{+},{\times})$. (By Lagrange's Four Squares Theorem.)
\end{example}

\noindent
If $\bA$ has uniform elimination of imaginaries, then every interpretation of $\bB$ in $\bA$ is homotopic to an injective interpretation of $\bB$ in $\bA$ \cite[Theorem~5.4.1]{Ho}. 
This applies to $\bA=\ZZ$, and in combination with the fact that every infinite definable subset of~$\ZZ^m$ is in definable bijection with~$\ZZ$, this yields:

\begin{lemma} \label{lem:bijective interpretation}
Every interpretation of an infinite structure $\bA$ in the ring $\ZZ$ of integers is homotopic to an injective interpretation of $\bA$ in $\ZZ$ whose domain is $\ZZ$. 
\end{lemma}

\noindent
So for example, if an infinite semiring $S$ is interpretable in $\ZZ$, then there are definable binary operations $\oplus$ and $\otimes$ on $\ZZ$ such that $(\ZZ,{\oplus},{\otimes})$ is isomorphic to $S$.

\begin{lemma} \label{lem:self-interpretations of Z}
Every self-interpretation of $\ZZ$ is homotopic to the identity interpretation.
\end{lemma}
\begin{proof}
Let $f\colon M\to\ZZ$ be a self-interpretation of $\ZZ$, where $M\subseteq \ZZ^m$. 
By Lemma~\ref{lem:bijective interpretation} we may assume that $f$ is bijective, $m=1$, and $M=\ZZ$. Hence the copy of $\ZZ$ interpreted in itself via $f$ has the form $Z=(\ZZ,{\oplus},{\otimes})$ where $\oplus$ and $\otimes$ are binary operations on $\ZZ$ definable in $\ZZ$. Let $0_Z$ and $1_Z$ denote the additive and multiplicative identity elements of the ring $Z$. The successor function $k\mapsto\sigma(k):=k\oplus 1_Z\colon\ZZ\to\ZZ$ in the ring $Z$ is definable in  $\ZZ$. Therefore the unique isomorphism $\ZZ\to Z$, given by $k\mapsto\sigma^k(0_Z)$ for $k\in\ZZ$, is computable, and hence definable in $\ZZ$; its inverse is $\bar{f}$.
\end{proof}

\noindent
Due to the  previous lemma, the task of checking that a pair of interpretations forms a bi-interpretation between $\bA$ and $\ZZ$ simplifies somewhat: a pair  $(f,g)$, where  $f\colon\bA\leadsto\ZZ$ and $g\colon\ZZ\leadsto\bA$, is
 a bi-interpretation between $\bA$ and $\ZZ$ iff $g\circ f\simeq\id_\bA$.

\begin{cor} \label{cor:all interpretations of Z are homotopic}
If $\bA$ and $\ZZ$ are bi-inter\-pretable, then any two interpretations of~$\ZZ$ in $\bA$ are homotopic.
\end{cor}
\begin{proof}
Suppose $(f,g)$, where $f\colon\bA\leadsto\ZZ$ and $g\colon\ZZ\leadsto\bA$, is a bi-interpretation between $\bA$ and $\ZZ$. Let $f'$ be an arbitrary interpretation $\bA\leadsto\ZZ$. Then $f\circ g$ and $f'\circ g$ are self-interpretations of $\ZZ$. Therefore $f\circ g\simeq f'\circ g$ by Lemma~\ref{lem:self-interpretations of Z} and thus $f\simeq f'$ by Corollary~\ref{cor:equiv of categories}.
\end{proof}

\subsection{Interpretations among rings} \label{sec:interpretations in rings}
{\it In this subsection we let $A$ be a ring.}\/
Familiar ring-theoretic constructions can be seen as interpretations: 

\begin{examples}\label{ex:ring interpretations}
\mbox{}
\begin{enumerate}
\item Let $S$ be a commutative semiring, and suppose $A$ is the Grothen\-dieck ring associated to $S$, that is,
$A=(S\times S)/E$ where $E$ is the equivalence relation on $S\times S$ given by 
$(x,y) E (x',y') :\Longleftrightarrow x+y'=x'+y$.
Then the natural map $S\times S\to A$ is an interpretation of $A$ in $S$.
\item For an ideal $I$ of $A$ which is definable in $A$ (as a subset of $A$), the residue morphism $A\to A/I$ is an interpretation of $A/I$ in $A$.
\item Suppose $A=A_1\times A_2$ is the direct product  of rings $A_1$, $A_2$. Then both factors~$A_1$ and $A_2$ are interpretable in $A$. (By the last example applied to the ideals $I_1=Ae_2$ respectively $I_2=Ae_1$, where $e_1=(1,0)$,  $e_2=(0,1)$.)
\item Let $S$ be a multiplicative subset of $A$ (that is, $1\in S$, $0\notin S$, and $S\cdot S\subseteq S$). 
Suppose $S$ is definable.  
Then the map $$M:=A\times S\to A[S^{-1}]\colon (a,s)\mapsto a/s$$ is an interpretation of the localization $A[S^{-1}]$ of $A$ at $S$  in $A$. Its kernel is the equivalence relation
$$(a,s)\sim (a',s') \qquad \Longleftrightarrow \qquad 
\exists t\in S\ \big( t\cdot (as'-a's)=0\big)$$
on $M$. 
In particular, if $A$ is an integral domain, then its fraction field is interpretable in $A$.
\end{enumerate}
\end{examples}

\noindent
Let~$S$ be a multiplicative subset of $A$. One says that $S$ is  \emph{saturated} if for all $a,b\in A$ we have $ab\in S$ if  $a\in S$ and $b\in S$. Equivalently, $S$ is saturated iff~$A\setminus S$ is a union of prime ideals of $A$. There is a smallest saturated multiplicative subset~$\overline{S}$ of~$A$ which contains $S$ (called the \emph{saturation} of $S$); here $A\setminus\overline{S}$ is the union of all prime ideals of $A$ which do not intersect $S$, and $A[S^{-1}]=A[\overline{S}{}^{-1}]$. (See \cite[Chapter~3, exercises]{AM}.)

\begin{lemma}\label{lem:interpret A[1/c]}
Suppose $A$ is a finite-dimensional noetherian Jacobson ring, and $c\in A$. Then $A[c^{-1}]$ is interpretable in $A$.
\end{lemma}

\begin{proof}
By Lemma~\ref{lem:def prime ideals}, the union of all prime ideals of $A$ which do not contain $c$ is definable in $A$, hence so is the saturation $\overline{S}$ of the multiplicative subset $c^\NN=\{c^n:n=0,1,2,\dots\}$ of $A$. Thus $A[c^{-1}]=A[\overline{S}{}^{-1}]$ is interpretable in $A$ by Examples~\ref{ex:ring interpretations},~(4).
\end{proof}

\noindent
Suppose $A$ is noetherian. Then
every finite ring extension $B$ of $A$ is interpretable in $A$:
choose generators $b_1,\dots,b_m$ of $B$ as an $A$-module, and 
let $K$ be the kernel of the surjective $A$-linear map $\pi\colon A^m\to B$ given by $(a_1,\dots,a_m)\mapsto \sum_i a_ib_i$.
Then~$K$ is a f.g.~$A$-submodule of $A^m$, hence definable in $A$.
The  multiplication map on $B$ may be encoded by a bilinear form on $A^m$. Thus $\pi$ is an interpretation of $B$ in $A$.

One says that $A$ has finite rank $n$ if each f.g.~ideal of $A$ can be generated by $n$ elements.
In this case, every submodule of $R^m$ can be generated by $mn$ elements~\cite{Cohen}.
Hence we obtain:

\begin{lemma}\label{lem:interpret finite extensions}
Suppose $A$ is noetherian of finite rank. Then the class of finite ring extensions of $A$ generated by $m$ elements as $A$-module is
uniformly interpretable in~$A$.
\end{lemma}

\noindent
This fact, together with its corollary below, are used in the proof of Theorem~\ref{thm:integral domains}.

\begin{cor}\label{cor:interpret finite extensions}
Suppose $A$ is noetherian of finite rank, and let $A'$ be a flat ring extension of $A$. Then
the class of rings of the form $A'\otimes_A B$, where $B$ is a finite ring extension of $A$ generated by $m$ elements as an $A$-module, is uniformly interpretable in $A'$.
\end{cor}
\begin{proof}
Let $B$ be a ring extension of $A$ generated as an $A$-module by $b_1,\dots,b_m$.
With $\pi$, $K$ as before we have an exact sequence
$$0\to K \xrightarrow{\ \subseteq\ } A^m\xrightarrow{\ \pi\ } B \to 0.$$
Tensoring with $A'$ yields an exact sequence
$$0\to A'\otimes_A K \longrightarrow (A')^m\xrightarrow{\ 1\otimes\pi\ } A'\otimes_A B \to 0.$$
The image of $K$ under $x\mapsto 1\otimes x$ generates the $A'$-module $A'\otimes_A K$, and the 
 extension of the bilinear form on the $A$-module $A^m$ which describes the ring multiplication on~$B$
to a bilinear form on the $A'$-module $(A')^m$ also describes the ring multiplication
on $A'\otimes_A B$.
\end{proof}

\noindent
We finish this subsection by recording a detailed proof of the well-known fact that all finitely generated rings are interpretable in $\ZZ$. The proof is a typical application of G\"odel coding in arithmetic, and we assume that  the reader is familiar with the basics of this technique; see, for example, \cite[Section~6.4]{Shoen}.
(Later in the paper, such routine coding arguments will usually only be sketched.)
Let $\beta$ be a G\"odel function, i.e., a function $\NN^2\to\NN$, definable in Peano Arithmetic (in fact, much weaker systems of arithmetic are enough), so that for any finite sequence~$(a_1,\dots,a_{n})$ of natural numbers there exists $a\in\NN$ such that 
$\beta(a,0)=n$ (the length of the sequence) and
$\beta(a,i)=a_i$ for $i=1,\dots,n$.
It is routine to construct from $\beta$ a function $\gamma\colon\NN^2\to\ZZ$ which is definable in $\ZZ$ and
which encodes finite sequences of integers, i.e.,
such that for each  
$(a_1,\dots,a_n)\in\ZZ^{n}$ there exists $a\in\NN$ with $\gamma(a,0)=n$ and $\gamma(a,i)=a_i$ for $i=1,\dots,n$.

\begin{lemma}\label{lem:poly ring}
Suppose $A$ is interpretable in $\ZZ$, and let $X$ be an indeterminate over~$A$. Then $A[X]$ is also interpretable in $\ZZ$.
\end{lemma}
\begin{proof}
For simplicity we assume that $A$ is infinite (the case of a finite $A$ being similar).
Let $g\colon \ZZ\to A$ be an injective interpretation of $A$ in $\ZZ$. (Lemma~\ref{lem:bijective interpretation}.) Let
$$N:=\big\{a\in\NN : \gamma(a,0) \geq 1, \text{ and }
\gamma(a,0)\geq 2 \Rightarrow \gamma(a,\gamma(a,0))\neq 0\big\}$$
be the set of codes of finite sequences $(a_0,\dots,a_n)\in\ZZ^{n+1}$ such that $a_n\neq 0$ if $n\geq 1$. Clearly $N$ is definable in $\ZZ$. It is easy to check that then the map
$$N\to A[X] \colon a \mapsto \sum_{i=0}^{\gamma(a,0)-1} g\big(\gamma(a,i+1)\big)\,X^i$$
is an injective interpretation of $A[X]$ in $\ZZ$. 
\end{proof}

\noindent
The previous lemma in combination with Examples~\ref{ex:ring interpretations},~(2) and (4)  yields:

\begin{cor}\label{cor:fg ring int in ZZ}
Every f.g.~ring and every localization of a f.g.~ring at a definable multiplicative subset is interpretable in $\ZZ$. 
\end{cor}

\begin{remarks*}[uniform interpretations in and of $\ZZ$] The following remarks are not used later in this paper.

\begin{enumerate}
\item The proof of Corollary~\ref{cor:fg ring int in ZZ} can be refined to  show that the class of f.g.~rings is {\it uniformly}\/ interpretable in $\ZZ$.
\item See \cite[Section~2]{Sc} for a proof that $\ZZ$ is uniformly interpretable in the class of infinite f.g.~fields. 
By (2) and (4) of Examples~\ref{ex:ring interpretations}, if $\mathfrak p$ is a prime ideal of $A$, then the fraction field of $A/\mathfrak p$ is interpretable in $A$. 
Using remark~(2) following Lemma~\ref{lem:def prime ideals} this implies that for each $n$, the class of infinite fields generated (as fields) by $n$ elements is uniformly interpretable in the class $\mathfrak A_n$ of infinite rings generated by $n$ elements. Hence
for each $n$, $\ZZ$ is uniformly interpretable in $\mathfrak A_n$. We do not know whether 
$\ZZ$ is uniformly interpretable in the class $\bigcup_n \mathfrak A_n$ of   infinite f.g.~rings. (This question was also asked in~\cite{Kh}.)
\end{enumerate}
\end{remarks*}

\subsection{Bi-interpretability with $\ZZ$}\label{sec:bi-int with ZZ}
In this subsection we deduce a few useful consequences of bi-interpretability with $\ZZ$.  

\medskip
\noindent
Suppose first that $\bA$ and $\ZZ$ are weakly bi-inter\-pretable, and
let $(f,g')$ be a weak bi-interpretation between $\bA$ and $\ZZ$. By Lemma~\ref{lem:bijective interpretation} there
is an injective interpretation $g\colon\ZZ\to A$ of $\bA$ in $\ZZ$ with $g\simeq g'$. By Lemma~\ref{lem:properties of homotopy},~(1) we have
$g\circ f\simeq g'\circ f\sim\id_\bA$, and by Lemma~\ref{lem:self-interpretations of Z} we have $f\circ g\simeq\id_\ZZ$.
Hence $(f,g)$ is a weak bi-interpretation between $\bA$ and $\ZZ$, and
if $(f,g')$ is even a bi-interpretation between $\bA$ and $\ZZ$, then so is $(f,g)$. 
Thus, if there is a weak bi-interpretation between $\bA$ and $\ZZ$ at all, then there is such a weak bi-interpretation $(f,g)$
where $g$ is a bijection $\ZZ\to A$; similarly with ``bi-interpretation'' in place of ``weak bi-interpretation.''

\medskip
\noindent
As a first application of these remarks, we generalize Lemma~\ref{lem:self-interpretations of Z} from $\ZZ$ to all structures bi-interpretable with $\ZZ$. 

\begin{cor}\label{cor:self-interpretations}
If $\bA$ and $\ZZ$ are bi-inter\-pretable, then every self-interpretation of~$\bA$ is homotopic to $\id_\bA$.
\textup{(}Hence if  $\bA$ and $\ZZ$ are bi-inter\-pretable, then any pair of interpretations $\bA\leadsto\ZZ$ and $\ZZ\leadsto\bA$ is a bi-interpretation between $\bA$ and $\ZZ$.\textup{)}
\end{cor}
\begin{proof}
Let $(f,g)$ be a bi-interpretation between $\bA$ and $\ZZ$ where $g$ is a bijection $\ZZ\to A$, and let $h\colon\bA\leadsto\bA$.
Then $f\circ h\circ g\simeq\id_\ZZ$ by Lemma~\ref{lem:self-interpretations of Z},  thus 
$h\circ g\simeq g$  by Lemma~\ref{lem:properties of homotopy} (and injectivity of $g$), and so
$h\simeq\id_\bA$ by Corollary~\ref{cor:equiv of categories}.
\end{proof}

\noindent
For the following corollary (used in the proof of Theorem~\ref{thm:integral domains} below), suppose we are given an 
isomorphism $\alpha\colon \bA\to\widetilde{\bA}$ of $\cL$-structures. Then $\alpha$ acts on
definable objects in the natural way. For example, if $i\colon M\to D$ ($M\subseteq A^m$) is an interpretation of~$\bD$ in $\bA$, then $i\circ\alpha\colon\alpha(M)\to D$ is an interpretation
of $\bD$ in $\widetilde{\bA}$, and $\alpha$ induces an isomorphism $\overline{\alpha}\colon i^*(\bD)\to (i\circ\alpha)^*(\bD)$.
Note that the underlying set of $(i\circ\alpha)^*(\bD)$ is $\alpha(M)/\ker(i\circ\alpha^{-1})$.

\begin{cor}\label{cor:def iso}
Let $i\colon \bA\leadsto\bD$ and $j\colon \bD\leadsto \bA$, and let $\widetilde{\bA}:=(j\circ i)^*(\bA)$ and $\alpha$ denote the inverse of the isomorphism
$\overline{j\circ i}\colon \widetilde{\bA}\to\bA$.
Suppose $\bD$ is bi-interpretable with $\ZZ$. Then $\overline{\alpha}\colon i^*(\bD)\to (i\circ\alpha)^*(\bD)$ is definable in $\bA$.
\end{cor}
\begin{proof}
One checks that $i$ induces an isomorphism $(i\circ\alpha)^*(\bD)\to (i\circ j)^*(\bD)$ which makes the diagram
$$\xymatrix{ (i\circ j)^*(\bD) \ar[r]^{\quad \overline{i\circ j}} & \bD \\ 
(i\circ\alpha)^*(\bD) \ar[u] \ar[r]^{\quad \overline{j\circ i}} & i^*(\bD) \ar[u]^{\overline{i}}}$$
commutative.
By Corollary~\ref{cor:self-interpretations}, the self-interpretation $i\circ j$ of $\bD$ is homotopic to~$\id_{\bD}$, that is,  $\overline{i\circ j}$ is definable in $\bD$, and so $\alpha=(\overline{j\circ i})^{-1}$ is definable in $\bA$.
\end{proof}

\noindent
Let now $f\colon\bA\leadsto\ZZ$ and $g\colon\ZZ\leadsto\bA$,  where $g$ is a bijection $\ZZ\to A$. We are going to analyze this situation in some more detail. Let $\tilde\ZZ=f^*(\ZZ)$ and $\tilde\bA=(g\circ f)^*(\bA)$. We have isomorphisms $\overline{g\circ f}\colon\tilde\bA\to\bA$ and $\bar{f}\colon\tilde\ZZ\to\ZZ$.
Note that as $g$ is bijective, we have
$$\tilde{A}=f^{-1}(\ZZ)/\ker(g\circ f)=M/\ker f,$$ 
so $\tilde{\bA}$ and $\tilde\ZZ$ have the same underlying set, and we have a commutative diagram
$$
\xymatrix{\ZZ \ar[rr]^g & & A \\
&\ar[ul]^{\bar{f}}\tilde\ZZ=\tilde A \ar[ur]_{\bar{g\circ f}}
}
$$
which shows the subtle fact that the identity map $\tilde\ZZ\to\tilde A$ is an interpretation of $\tilde\bA$ in $\tilde\ZZ$.

\medskip
\noindent
For the next lemma, we say that a structure with the same universe as $\bA$  is {\it interdefinable} with $\bA$
if both structures have the same definable sets.

\begin{lemma}\label{lem:oplus and otimes}
The following are equivalent:
\begin{enumerate}
\item $\bA$ is bi-interpretable with $\ZZ$;
\item $\bA$ is weakly bi-interpretable with $\ZZ$;
\item there are binary operations $\oplus$ and $\otimes$ on $A$ such that
\begin{enumerate}
\item $(\ZZ,{+},{\times})\cong (A,{\oplus},{\otimes})$;
\item $(A,{\oplus},{\otimes})$ is interdefinable with $\bA=(A,\dots)$.
\end{enumerate}
\end{enumerate}
\end{lemma}
\begin{proof}
It is clear that if we have binary operations $\oplus$ and $\otimes$ on $A$ satisfying conditions (a) and (b) in (3), then $(f,g)$, where $f\colon A\to\ZZ$ is the unique isomorphism $(A,{\oplus},{\otimes})\to (\ZZ,{+},{\times})$ and $g=f^{-1}$, is a bi-interpretation between $\bA$ and $\ZZ$.
Conversely, suppose $\bA$ is weakly bi-interpretable with $\ZZ$ via a weak bi-interpretation~$(f,g)$ where $g$ is a bijection $\ZZ\to A$. Let $\alpha$ be an isomorphism $\tilde\bA=(g\circ f)^*(\bA)\to\bA$, definable in $\bA$.
Let $\oplus$ and $\otimes$ be the binary operations on $A$ such that $\beta$ is an isomorphism $(\tilde{\ZZ},{+},{\times})\to (A,{\oplus},{\otimes})$. (Recall that $\tilde\ZZ=\tilde A$ as sets.)
The operations $\oplus$ and $\otimes$ are then definable in $\bA$;
conversely, since $\alpha$ is also an isomorphism of $\cL$-structures $\tilde{\bA}\to\bA$ and
the identity $\tilde{\ZZ}\to\tilde{A}$
is an interpretation of $\tilde{\bA}$ in $\tilde{\ZZ}$, the interpretations of the function and relation symbols of $\cL$ in $\bA$ are definable in $(A,{\oplus},{\otimes})$.
\end{proof}

\noindent
As an illustration of this analysis, next we show:

\begin{lemma}
\label{lem:biNdeforb}
Suppose $\bA$ is bi-interpretable with $\ZZ$.
Let $\Phi\colon A\to A$ be definable,
and let $a\in A$. Then the orbit
$$\Phi^\NN(a):=\big\{\Phi^{\circ n}(a):n=0,1,2,\dots\big\}\qquad \text{\textup{(}$\Phi^{\circ n}=$ $n$th iterate of $\Phi$\textup{)}}$$
of $a$ under $\Phi$ is definable. 
\end{lemma}

\begin{proof}
Via G\"{o}del coding of sequences, it is easy to see that the lemma holds if~$\bA=\ZZ$.
In the general case, suppose $\otimes$, $\oplus$ are binary operations on $A$ satisfying conditions (a) and (b) of the previous lemma. Then the map $\Phi$ is also definable in~$(A,\oplus,\otimes)$, hence $\Phi^\NN(a)$ is definable in~$(A,\oplus,\otimes)$, and thus also in $\bA$. 
\end{proof}

\noindent
Let us note two consequences of Lemma~\ref{lem:biNdeforb} for rings.

\begin{cor}
\label{cor:biNdefZ}
Let $A$ be a ring of characteristic zero which is bi-interpretable with $\ZZ$. Then the natural image of $\ZZ$ in $A$ is definable as a subring of $A$.
\end{cor}
\begin{proof}
The image of $\ZZ$ is $(x\mapsto x+1)^\NN(0) \cup (x\mapsto x-1)^{\NN}(0)$.   Apply Lemma~\ref{lem:biNdeforb}.
\end{proof}

\begin{cor}
\label{cor:biNdefpower}
Let $A$ be a ring which is bi-interpretable with $\ZZ$, and $a \in A$. Then the set $a^\NN:=\{ a^n : n=0,1,2,\dots \}$ of powers of $a$ is definable. 
\end{cor}
\begin{proof}
The set $a^\NN$ is $(x \mapsto a x)^\NN(1)$.  Apply Lemma~\ref{lem:biNdeforb}.
\end{proof}

\noindent
Here is a refinement of Lemma~\ref{lem:biNdeforb}. For an interpretation $f$ of $\ZZ$ in $\bA$, 
the restriction of $f$ to a map $f^{-1}(\NN)\to\NN$ is an interpretation of $\NN$ in $\bA$, and
by abuse of notation we denote the copy of $\NN$ interpreted in $\bA$ via this interpretation by~$f^*(\NN)$, and we write $n\mapsto\overline{n}$ for the inverse of the isomorphism $f^*(\NN)\to\NN$.

\begin{lemma} 
Suppose $\bA$ is bi-interpretable with $\ZZ$, and let $f\colon \bA\leadsto\ZZ$.
Let $\Phi\colon A\to A$ be definable.
Then the map $$(a,\overline{n}) \mapsto \Phi^{\circ n}(a)\colon A\times f^*(\NN)\to A$$ is definable.
\end{lemma}

\begin{proof}
Let $g\colon\ZZ\to A$ be an injective interpretation $\ZZ\leadsto\bA$; then $(f,g)$ is a bi-interpretation between $\bA$ and $\ZZ$. (Corollary~\ref{cor:self-interpretations}.)
Let $\oplus$ and $\otimes$ be the binary operations on $A$ making $g$ an isomorphism $(\ZZ,{+},{\times})\to (A,{\oplus},{\otimes})$. Then $\oplus$ and~$\otimes$ satisfy (a) and (b) in Lemma~\ref{lem:oplus and otimes} (by the proof of said lemma). The map $\Phi$ is definable in $(A,{\oplus},{\otimes})$, and thus
$$(a,b)\mapsto  \Phi^{\circ g^{-1}(b)}(a)\colon A\times g(\NN)\to A $$
is definable in $(A,{\oplus},{\otimes})$, and hence also in $\bA$. Therefore, since $[g\circ f=\id_A]$ is definable in $\bA$, so is (the graph of) the map $$(a,b) \mapsto \Phi^{\circ f(b)}(a)\colon A\times f^{-1}(\NN)\to A.$$
The lemma follows.
\end{proof}

\noindent
The last lemma immediately implies:

\begin{cor}\label{cor:power map}
Let $A$ be a ring which is bi-interpretable with $\ZZ$, and let $f\colon A\leadsto\ZZ$. Then the map
$$(a,\overline{n}) \mapsto a^{n}\colon A\times f^*(\NN) \to A$$
is definable. 
\end{cor}

\subsection{A test for bi-interpretability with $\ZZ$}\label{sec:Nies}
Suppose that $(f,g)$ is a weak bi-interpretation between $\bA$ and $\ZZ$ where $g$ is a bijection $\ZZ\to A$. As remarked in the previous subsection, we then have $f^*(\ZZ)=(g\circ f)^*(A)$ as sets, so the inverse of any definable
isomorphism $(g\circ f)^*(\bA)\to\bA$ (which exists since $g\circ f\sim\id_\bA$) is a bijection $A\to f^*(\ZZ)$ which is definable in $\bA$. The following proposition is a partial converse of this observation:

\begin{prop} \label{prop:Nies}
Suppose that $\bA$ is f.g.~and the language~$\cL=\cL_\bA$ of $\bA$ is finite.
Let $f\colon \bA\leadsto\ZZ$ and $g\colon\ZZ\leadsto\bA$.
Suppose also that there exists an injective map~${A\to f^*(\ZZ)}$ which is definable in $\bA$.
Then $(f,g)$ is a weak bi-interpretation between~$\bA$ and $\ZZ$.
\end{prop}

\noindent
An important consequence of this proposition (and Lemma~\ref{lem:oplus and otimes}) is that under reasonable assumptions on $\bA$ and $\cL$, establishing bi-interpretability of $\bA$ with~$\ZZ$ simply amounts to showing that $\bA$ is interpretable in $\ZZ$, and 
$\ZZ$ is interpretable in~$\bA$ in such a way that there is a definable way to index the elements of $\bA$ with elements of the copy of $\ZZ$ in $\bA$:

\begin{cor}[Nies] \label{cor:Nies}
If $\bA$ is f.g.~and $\cL$ is finite, then the following are equivalent:
\begin{enumerate}
\item $\bA$ is \textup{(}weakly\textup{)} bi-interpretable with $\ZZ$;
\item $\bA$ is interpretable in $\ZZ$, and there is an interpretation $f$ of $\ZZ$ in $\bA$ and an injective definable map $A\to f^*(\ZZ)$.
\end{enumerate}
\end{cor}

\noindent
A proof of this corollary of Proposition~\ref{prop:Nies} is sketched in \cite[Proposition~7.12]{Nies}. However, we feel that a more detailed argument is warranted. (Also note that loc.~cit.~does not assume $\bA$ to be f.g.)
Before we give a proof of Proposition~\ref{prop:Nies}, we show two auxiliary facts:

\begin{lemma}\label{lem:claim 1}
Let $f$ be a self-interpretation of $\bA$ which is homotopic to the identity. Then every set $X\subseteq f^*(A)^n$ which is definable in $\bA$ is definable in $f^*(\bA)$.
\end{lemma}
\begin{proof}
Suppose $f$ is given by $M\to A$ where $M\subseteq A^m$ is definable.
Let $\xi(x_1,\dots,x_n)$ be an $\cL$-formula, possibly involving parameters, which defines $X$ in~$\bA$; that is, for all
$a\in M^n\subseteq (A^m)^n$ we have
$$\bA\models\xi(a) \quad\Longleftrightarrow\quad \overline{a}\in X.$$
By hypothesis, the isomorphism $\overline{f}\colon f^*(\bA)\to\bA$ is definable in $\bA$. Let $\varphi(x,y)$ define
its graph; that is, for $a\in M$ and $b\in A$ we have
$$\bA\models\varphi(a,b)\quad \Longleftrightarrow\quad \overline{f}(\overline{a})=b.$$
Set
$$\xi^*(y_1,\dots,y_n):=\exists x_1\cdots\exists x_n\left(\xi(x_1,\dots,x_n) \ \&\ \bigwedge_{i=1}^n \varphi(x_i,y_i)\right).$$
Then for $a\in M^n$ we have
\begin{align*}
f^*(\bA)\models\xi^*(\overline{a})	&\Longleftrightarrow \bA\models \xi^*(\overline{f}(\overline{a})) \\
									&\Longleftrightarrow \bA\models\exists x_1\cdots\exists x_n\left(\xi(x_1,\dots,x_n) \ \&\ \bigwedge_{i=1}^n\varphi\big(x_i,\overline{f}(\overline{a_i})\big)\right) \\
									&\Longleftrightarrow \bA\models\xi(a) 
									\ \Longleftrightarrow\ \overline{a}\in X.
\end{align*}
hence $\xi^*$ defines $X$ in $f^*(\bA)$.
\end{proof}

\begin{lemma}\label{lem:claim 2}
Let $\bA$ be a f.g.~structure in a finite language. Then any two interpretations of $\bA$ in $\ZZ$ are homotopic.
\end{lemma}
\begin{proof}
Let  $f,g\colon \ZZ\leadsto\bA$; by Lemma~\ref{lem:bijective interpretation} we may assume that $f$ and $g$ are injective with domain $\ZZ$. Let $a_1,\dots,a_n\in A$ be generators for $\bA$ and let $b_i:=\overline{f}{}^{-1}(a_i)$,
$c_i:=\overline{g}{}^{-1}(a_i)$, for $i=1,\dots,n$, be the corresponding elements of $f^*(A)$ and $g^*(A)$, respectively.
The unique isomorphism $f^*(\bA)\to g^*(\bA)$ given by $b_i\mapsto c_i$ ($i=1,\dots,n$) is relatively computable and hence definable in $\ZZ$.
\end{proof}

\noindent
We now show Proposition~\ref{prop:Nies}. Thus, let $f\colon \bA\leadsto\ZZ$ and $g\colon\ZZ\leadsto\bA$, and let
$\phi\colon A\to f^*(\ZZ)$ be an injective map, definable in $\bA$. By Lemma~\ref{lem:self-interpretations of Z} we have $f\circ g\simeq\id_\ZZ$, so it is enough to show that  $g\circ f\sim\id_\bA$.
Recall that $g$ induces an isomorphism $(f\circ g)^*(\ZZ)\to f^*(\ZZ)$, and thus,
pulling back $\phi$ under $\overline{g}$ we obtain a $g^*(\bA)$-definable injective map
$g^*(\phi) \colon g^*(A) \to (f\circ g)^*(\ZZ)$ making the diagram
$$\xymatrix@C=5em{A \ar[r]^\phi & f^*(\ZZ) \\
g^*(A) \ar[r]^{g^*(\phi)} \ar[u]^{\overline{g}\ } & (f\circ g)^*(\ZZ) \ar[u]
}$$
commute.
We make its image $g^*(\phi)\big(g^*(A)\big)$ the universe of an $\cL$-structure, which we denote by $g^*(\phi)\big(g^*(\bA)\big)$, such that $g^*(\phi)$ becomes an isomorphism.
Note that both the underlying set $g^*(\phi)\big(g^*(A)\big)$ as well as the interpretations of the function and relation symbols of $\cL$ in this structure are definable in $g^*(\bA)$, hence in $\ZZ$, and so,
by Lemma~\ref{lem:claim 1}, also in $(f\circ g)^*(\ZZ)$.
Thus we obtain an interpretation $h$ of $\bA$ in $(f\circ g)^*(\ZZ)$
with $h^*(\bA)=g^*(\phi)\big(g^*(\bA)\big)$. On the other hand, suppose $g$ is given by $N\to A$ where $N\subseteq\ZZ^n$; then setting 
$$N':=(\overline{f\circ g})^{-1}(N), \qquad g':=g\circ (\overline{f\circ g})\upharpoonright N',$$ 
we have another interpretation $g'\colon (f\circ g)^*(\ZZ)\leadsto\bA$. 
By Lemma~\ref{lem:claim 2}, the interpretations $h$ and  $g'$ are homotopic. Thus we have an isomorphism $h^*(\bA)\to(g')^*(\bA)$ which is definable in $(f\circ g)^*(\ZZ)$, and hence in $g^*(\bA)$. Composing this isomorphism with the isomorphism $g^*(\phi)\colon g^*(\bA)\to h^*(\bA)$, which is also definable in $g^*(\bA)$, 
yields an isomorphism $g^*(\bA)\to (g')^*(\bA)$ which is definable in $g^*(\bA)$.
It is routine to verify that the isomorphism $\overline{g}\colon g^*(\bA)\to\bA$ maps
the domain $N'$ of $g'$ bijectively onto the domain $f^{-1}(N)$ of $g\circ f$, and that this bijection induces a bijection $(g')^*(A)\to (g\circ f)^*(A)$ which is compatible
with $\overline{g'}$ and $\overline{g\circ f}$, and hence an isomorphism $(g')^*(\bA)\to (g\circ f)^*(\bA)$.
Thus our definable isomorphism $g^*(\bA)\to (g')^*(\bA)$ gives rise to an isomorphism $\bA\to (g\circ f)^*(\bA)$ which fits into the commutative diagram
$$\xymatrix{\bA \ar[r] & (g\circ f)^*(\bA) \\ 
g^*(\bA) \ar[u]^{\overline{g}} \ar[r] & (g')^*(\bA) \ar[u]}$$
and which is definable in $\bA$, as required. \qed

\subsection{Quasi-finite axiomatizability}
{\it In this subsection we assume that $\cL$ is finite and $\bA=(A,\dots)$ is f.g.}\/
We say that an $\cL$-formula $\varphi_\bA(x_1,\dots,x_n)$ is a {\it QFA formula} for $\bA$ 
with respect to the system of generators $a_1,\dots,a_n$ of $\bA$ if the following holds: if~$\bA'$ is any f.g.~$\cL$-structure and $a_1',\dots,a_n'\in A'$, then $\bA'\models\varphi_\bA(a_1',\dots,a_n')$ iff there is an isomorphism $\bA\to \bA'$ with $a_i\mapsto a_i'$ for $i=1,\dots,n$.
Any two QFA formulas for $\bA$ with respect to the same system of generators of $\bA$ are equivalent in $\bA$. Moreover:

\begin{lemma}
Let $\varphi_\bA(x_1,\dots,x_n)$ be a QFA formula for $\bA$ with respect to the system of
generators $a_1,\dots,a_n$ of $\bA$. Then for each system of generators $b_1,\dots,b_m$ of $\bA$
there is a QFA formula for $\bA$ with respect to $b_1,\dots,b_m$.
\end{lemma}
\begin{proof}
For notational simplicity we assume that $m=n=1$ (the general case is only notationally more complicated). 
Let $b$ be a generator for $A$. Let $s(x)$, $t(y)$ be $\cL$-terms such that $a=t^{\bA}(b)$ and $b=s^{\bA}(a)$. Put $\psi(y):=\varphi(t(y))\wedge y=s(t(y))$. Then $\psi$ is a QFA formula for $\bA$ with respect to $b$.
\end{proof}

\noindent
A QFA formula for $\bA$ is a formula $\varphi_\bA(x_1,\dots,x_n)$ which is QFA for $\bA$ with respect to
some system of generators $a_1,\dots,a_n$ of $\bA$.
Note that if there is a QFA formula~$\varphi_\bA(x_1,\dots,x_n)$ for $\bA$, then~$\bA$ is {\it quasi-finitely axiomatizable}\/ ({\it QFA}\/), i.e., there is an $\cL$-sentence~$\sigma$ such that for every $\cL$-structure $\bA'$, we have $\bA'\models\sigma$ iff~${\bA\cong\bA'}$. (Take $\sigma=\exists x_1\cdots\exists x_n \varphi_\bA$.)
If $A$ is finite, then there clearly is a QFA formula for~$\bA$.
In this subsection we are going to show (see~\cite[Theorem~7.14]{Nies}):
 
\begin{prop}\label{prop:QFA}
If $\bA$ is  bi-interpretable with $\ZZ$, then  there is a QFA~formula for $\bA$.
\end{prop}

\noindent
Before we give the proof of this proposition, we make some observations. For these, we assume that the hypothesis of Proposition~\ref{prop:QFA} holds, that is,  that we have binary operations $\oplus$ and $\otimes$ on $A$ as in (a) and (b) of Lemma~\ref{lem:oplus and otimes}. We take $\cL$-formulas $\varphi_{\oplus}(x_1,x_2,y,z)$ and $\varphi_{\otimes}(x_1,x_2,y,z)$, where $z=(z_1,\dots,z_k)$ for some $k\in\NN$, and for each
function symbol~$f$ of~$\cL$, of arity $m$, and for each function symbol $R$ of~$\cL$, of arity $n$,
we take formulas $\varphi_f(x_1,\dots,x_m,y)$  and $\varphi_R(x_1,\dots,x_n)$ in the language of rings, and some
 $c\in A^{k}$,   such that
\begin{enumerate}
\item $\varphi_{\oplus}(x_1,x_2,y,c)$ and  $\varphi_{\otimes}(x_1,x_2,y,c)$ define $\oplus$ and $\otimes$ in $\bA$, respectively;
\item $\varphi_f(x_1,\dots,x_m,y)$ and $\varphi_R(x_1,\dots,x_n)$ define $f^\bA$ and $R^\bA$, respectively,  in $(A,{\oplus},{\otimes})$.
\end{enumerate}
We now let $\alpha_0(z)$ be an $\cL$-formula which expresses (1) and (2) above, i.e., such that $\bA\models\alpha_0(c)$, and 
for all $\cL$-structures $\bA'$ and $c'\in (A')^{k}$ such that $\bA'\models\alpha_0(c')$, 
\begin{enumerate}
\item $\varphi_{\oplus}(x_1,x_2,y,c')$ and  $\varphi_{\otimes}(x_1,x_2,y,c')$ define binary operations $\oplus'$ and $\otimes'$, respectively, on $A'$; and
\item $\varphi_f(x_1,\dots,x_m,y)$ and $\varphi_R(x_1,\dots,x_n)$ define $f^{\bA'}$ and $R^{\bA'}$, respectively, in $(A',{\oplus'},{\otimes'})$, for all function symbols $f$ and relation symbols $R$ of $\cL$.
\end{enumerate}
We also require that if $\bA'\models\alpha_0(c')$, then
\begin{itemize}
\item[(3)] $(A',{\oplus'},{\otimes'})$ is a ring which is a model of a sufficiently large (to be specified) finite fragment of $\operatorname{Th}(\ZZ)$.
\end{itemize}
The ring $(A',{\oplus'},{\otimes'})$ may be {\it non-standard,} that is, not isomorphic to~$(\ZZ,{+},{\times})$. However, 
choosing the finite fragment of arithmetic in (3) appropriately, we can ensure that we have a unique embedding  $(\ZZ,{+},{\times})\to (A',{\oplus'},{\otimes'})$. From now on we assume that $\alpha_0$ has been chosen in this way. 
Additionally we can choose $\alpha_0$ so that finite objects, such as $\cL$-terms and finite sequences of elements of $A'$, can be encoded in~$(A',\oplus',\otimes')$. This can be used to 
uniformly define term functions in $\bA'$, and leads to a proof of the following (see \cite[Claim~7.15]{Nies} for the details):

\begin{lemma}
There is an $\cL$-formula $\alpha(z)$, which logically implies $\alpha_0(z)$, such that
$\bA\models\alpha(c)$, and whenever $\bA'$ is a f.g.~$\cL$-structure and   $c'\in (A')^{k}$, then
$\bA'\models\alpha(c')$ iff $(A',{\oplus'},{\otimes'})$ is standard.
\end{lemma}

\noindent
Let now $t=(t_1,\dots,t_n)$ be a tuple of constant terms in the  language of rings. Given $\bA'\models\alpha_0(c')$, we denote by $t(c')=\big(t_1(c'),\dots,t_n(c')\big)$ the tuple containing the interpretations of the $t_i$ in the ring $(A',{\oplus'},{\otimes'})$. We also let $\alpha$ be as in the previous lemma.

\begin{lemma}
Let $\bA'$ is a f.g.~$\cL$-structure and $c'\in (A')^{k}$ with $\bA'\models\alpha(c')$. Then the orbit of $t(c')$ under $\Aut(\bA')$ is $\emptyset$-definable in $\bA'$.
\end{lemma}
\begin{proof}
We claim that for $a'=(a_1',\dots,a_n')\in (A')^n$ we have
$$\sigma(t(c'))=a' \text{ for some $\sigma\in\Aut(\bA')$} \ \Longleftrightarrow \  
t(c'')=a' \text{ for some $c''$ with $\bA'\models\alpha(c'')$.}$$
Here the forward direction is clear. For the backward direction suppose
$\bA'\models\alpha(c'')$, and let $\oplus''$, $\otimes''$ denote the binary operations on $A'$ defined by 
$\varphi_{\oplus}(x_1,x_2,y,c'')$,  $\varphi_{\otimes}(x_1,x_2,y,c'')$, respectively. We then have a unique isomorphism $(A',\oplus',\otimes')\to (A',\oplus'',\otimes'')$. This isomorphism maps $t(c')$ onto $t(c'')$, and is also an automorphism of $\bA'$, by condition (2) in the description of $\alpha_0$ above.
This shows the claim, and hence the lemma.
\end{proof}

\begin{proof}[Proof of Proposition~\ref{prop:QFA}]
Let $a_1,\dots,a_n\in A$ generate $\bA$, and let $t_1,\dots,t_n$ be the constant terms in the ring language corresponding to the images of $a_1,\dots,a_n$, respectively, under the isomorphism $(A,{\oplus},{\otimes})\to (\ZZ,{+},{\times})$. Then for each f.g.~$\cL$-structure $\bA'$ and $a_1',\dots,a_n'\in A'$, there is an isomorphism $\bA\to \bA'$ with $a_i\mapsto a_i'$ ($i=1,\dots,n$) iff there is some $c'$ such that $\bA'\models\alpha(c')$ and an automorphism of~$\bA'$ with $t_i(c')\mapsto a_i'$ ($i=1,\dots,n$).  By the lemma above, the latter condition is definable.
\end{proof}

\section{Integral Domains}\label{sec:integral domains}

\noindent
The goal of this section is to show the following theorem:

\begin{theorem}\label{thm:integral domains}
Every infinite f.g.~integral domain is bi-interpretable with $\ZZ$.
\end{theorem}

\noindent
Combining this theorem with Proposition~\ref{prop:QFA} immediately yields:

\begin{cor}\label{cor:QFA int domains}
Every f.g.~integral domain has a QFA formula.
\end{cor}

\noindent
Although Theorem~\ref{thm:integral domains} can be deduced from the main result of~\cite{Sc} (and is unaffected by the error therein), we prefer to start from scratch and give a self-contained proof of this fact. 

\medskip
\noindent
{\it In the rest of this section we let $A$ be an integral domain with fraction field~$K$.}\/ 

\subsection{Noether Normalization and some of its applications}
Our main tool is the Noether Normalization Lemma in the following explicit form:

\begin{prop}\label{prop:NN}
Suppose that $A$ is a f.g.~$D$-algebra, where  $D$ is a subring of~$A$. Then there are a 
nonzero $c\in D$ and $x_1,\dots,x_n\in A$, algebraically independent over~$D$, such that
$A[c^{-1}]$ is a finitely generated $D[c^{-1},x_1,\dots,x_n]$-module.
\end{prop}

\noindent
If the field $K$ is f.g., we define the \emph{arithmetic}\/ (or \emph{Kronecker}\/) dimension of $A$ as
$$
\operatorname{adim}(A) := \begin{cases}  \trdeg_{\QQ} (K) + 1 & \text{ if } \cha(A) = 0, \\	
 				\trdeg_{\FF_p}(K)  & \text{ if } \cha(A) = p > 0.
 				\end{cases}
 				$$
As a consequence of Proposition~\ref{prop:NN}, if the integral domain $A$ is f.g., then $\operatorname{adim}(A)$ equals the Krull dimension $\dim(A)$ of $A$.

\medskip
\noindent
Proposition~\ref{prop:NN} is particularly useful when combined with the following fact (a basic version of Grothendieck's ``generic flatness lemma''); see \cite[Theorem~2.1]{Vasconcelos}.

\begin{prop}\label{prop:generic freeness}
Suppose that $A$ is a f.g.~$D$-algebra, where  $D$ is a subring of~$A$. Then there is some $c\in D\setminus\{0\}$ such that $A[c^{-1}]$ is a free $D[c^{-1}]$-module.
\end{prop}

\noindent
The integral domain $A$ is said to be \emph{Japanese} if the integral closure of $A$ in a finite-degree field extension of $K$ is always a finitely generated $A$-module.
Every finitely generated integral domain is Japanese; see \cite[Theorem~36.5]{Nagata}.

\begin{lemma}\label{lem:algebraic => finite}
Let $D$ be a Japanese noetherian  subring of~$A$, 
$x_1,\dots,x_n\in A$ be algebraically independent over $D$, and suppose that $A$ is finite over   $R=D[x_1,\dots,x_n]$. 
Then every subring of $A$ which contains~$D$ and is algebraic over~$D$, is finite over~$D$.
\end{lemma}
\begin{proof}
Let $B$ be a subring of $A$ with $D\subseteq B$ which is algebraic over $D$.
We first show that $B$ is integral over $D$. Let $b\in B$.
Then $b$ is integral over $R$, that is,
satisfies an equation of the form
$f(b) = 0$ for some monic polynomial $f \in R[Y]$
in the indeterminate~$Y$.  With $\alpha=(\alpha_1,\dots,\alpha_n)$ ranging over $\NN^n$, write
$$f = \sum_{\alpha}  x^\alpha f_\alpha(Y)\qquad\text{
where $x^\alpha = x_1^{\alpha_1} \cdots x_n^{\alpha_n}$  and 
$f_\alpha(Y) \in D[Y]$.}$$
Since $B$ is algebraic over $D$, $x_1,\dots,x_n$ remain algebraically independent over $B$.  Hence, $f_\alpha(b) = 0$ 
for all $\alpha$.  In particular, $f_{{\mathbf 0}} (b) = 0$ and 
the polynomial $f_{{\mathbf 0}}$ is monic.  Therefore, $b$ is integral 
over $D$. 

Next we note that $K=\Frac(A)$ is a finite-degree field extension
of $L:=\Frac(R)$. 
Again because 
$x_1, \ldots, x_n$ are 
algebraically independent over $B$, each $D$-linearly independent sequence $b_1,\dots,b_m$ of elements of $B$ is also $R$-linearly independent and hence $L$-linearly independent, and so $m\leq [K:L]$.
Take $D$-linearly independent $b_1,\dots,b_m\in B$ with $m$ maximal, and set $M:=D[b_1,\dots,b_m]$.
Then $M$ is a f.g.~$D$-submodule of $B$, and the quotient module $B/M$ is torsion.
Hence $\Frac(B)=\Frac(M)$, and the degree of $\Frac(M)$ over $\Frac(D)$ is finite.
Therefore the integral closure of $D$ in $\Frac(B)$ is a f.g.~$D$-module; since this integral closure contains $B$ and $D$ is noetherian, 
$B$ is a f.g.~$D$-module as well.
\end{proof}

\noindent
With the following lemma we establish a basic result in commutative 
algebra.  It bears noting here that our hypothesis that the subring in 
question has arithmetic dimension one is necessary. It is not 
hard to produce non-finitely generated two-dimensional subrings of 
finitely generated integral domains.

\begin{lemma}
\label{onefg}
Suppose $A$ is finitely generated. Then
every subring of $A$ of arithmetic dimension~$1$  is finitely generated.	 
\end{lemma}
\begin{proof}
Let $B$ be a subring of $A$ with $\operatorname{adim}(B)=1$.
If $\cha(A) = 0$, then let $D := {\mathbb Z} \subseteq B$.
If $\cha(A) = p > 0$, pick some $t \in B$ transcendental over 
${\mathbb F}_p$ and  set $D := {\mathbb F}_p[t]  \subseteq B$.  
By Proposition~\ref{prop:NN} we can find some $c \in D \setminus 
\{ 0 \}$ and $x_1, \ldots, x_n \in A$ which are algebraically
independent over $D$ and for which $A[c^{-1}]$ is a 
finite integral extension of $D[c^{-1},x_1,\ldots,x_n]$.  
Since  $\operatorname{adim}(B) = 1$, 
$B[c^{-1}]$ is algebraic over $D[c^{-1}]$.
Hence by Lemma~\ref{lem:algebraic => finite} applied to $A[c^{-1}]$,~$D[c^{-1}]$
in place of~$A$,~$D$, respectively, $B[c^{-1}]$ is a finitely generated $D[c^{-1}]$-module.
 Choose generators $y_1, \ldots, y_m$
of $B[c^{-1}]$ as $D[c^{-1}]$-module.  Scaling by a sufficiently 
high power of $c$, we may assume that each~$y_i$ belongs to~$B$ and is integral over $D$.   
Then setting $R:=D[y_1, \ldots, y_m]$ we have $R \subseteq B \subseteq B[c^{-1}] = 
R[c^{-1}]$.  
By Corollary~\ref{cor:fg criterion}, $B$ is a f.g.~$R$-algebra, hence also a f.g.~ring.    	
\end{proof}

\subsection{Proof of Theorem~\ref{thm:integral domains}}
{\it In this subsection we assume that $A$ is f.g.}\/
We begin by showing that as an easy consequence of results of J. Robinson,
R. Robinson, and Rumely, each f.g.~integral domain of dimension~$1$ is bi-interpretable
with ${\mathbb Z}$. 
We deal with characteristic zero and positive characteristic
in separate lemmata:

\begin{lemma}
\label{biinterporder}	
Suppose that $\cha(A)=0$ and $\dim(A)=1$. Then 
$A$ is bi-interpretable with ${\mathbb N}$.
\end{lemma}

\begin{proof}
The field extension $K|\QQ$ is finite; set $d := [K : {\mathbb Q}]$. 
J.~Robinson showed~\cite{JRob} that the ring ${\mathcal O}_K$ 
of algebraic integers in $K$ is definable in $K$, 
and the subset~${\mathbb Z}$ is definable
in ${\mathcal O}_K$; hence ${\mathbb Z}$ is definable in $A$. 
Take an integer $c > 0$ such that $A \subseteq {\mathcal O}_K 
[\frac{1}{c}]$.
The map $n \mapsto c^n \colon {\mathbb N} \to  {\mathbb N}$
is definable in $A$, and so is the map $\nu\colon A \to {\mathbb N}$ which
associates to $a \in A$ the smallest $n := \nu(a) \in {\mathbb N}$
such that $c^n a \in  {\mathcal O}_K$. Fixing a
basis $\omega_1, \ldots, \omega_d \in  {\mathcal O}_K$ 
of the free ${\mathbb Z}$-module ${\mathcal O}_K$, 
we obtain a definable injective
map $A \hookrightarrow {\mathbb Z}^d \times {\mathbb N}$ 
by associating to $a \in A$ the tuple $\big(k_1(a), \ldots, k_d(a), \nu(a)\big)$,
where $\big(k_1(a), \ldots, k_d(a)\big)$ 
is the unique element of $\ZZ^d$ 
such that $c^{\nu(a)} a = \sum_{i=1}^d  k_i(a) \omega_i$. Hence,
$A$ is bi-interpretable with~${\mathbb Z}$ by Corollary~\ref{cor:Nies}.	
\end{proof}

\begin{lemma}
\label{biinterpposone}
Suppose that $\cha(A)>0$ and $\dim(A)=1$. Then   $A$ is bi-interpretable with ${\mathbb N}$.
\end{lemma}

\begin{proof}
Let $p := \operatorname{char}(A)$, and by Noether Normalization take some 
$t \in A$, transcendental over ${\mathbb F}_p$, such that 
$A$ is a finite extension of ${\mathbb F}_p[t]$.  Rumely~\cite[Theorem~2]{Rum} 
showed that $k[t]$ is definable in $K$, where
$k$ is the constant field of~$K$ (i.e., the relative algebraic closure of 
${\mathbb F}_p$ in $K$).  R.~Robinson~\cite{RRob} specified a formula~$\tau(x,y)$ with the property that for each finite field ${\mathbb F}$, 
$\tau(x,t)$ defines the set~$t^{\mathbb N}$ in 
${\mathbb F}[t]$. It follows that the binary operations on 
$t^{\mathbb N}$ making $n \mapsto t^n \colon {\mathbb N} \to t^{\mathbb N}$
an isomorphism of semi\-rings are definable in ${\mathbb F}[t]$. 	Thus, 
the inverse of this isomorphism is an interpretation~${\mathbb F}[t] \leadsto {\mathbb N}$.  Let $N = N_p$ be the set of natural 
numbers of the form $n = \prod_{i \geq 1} p_i^{n_i}$
with $n_i \in \{ 0, 1, \ldots, p-1 \}$, all but finitely many $n_i = 0$, and~$p_i$ is the $i^\text{th}$ prime number.   Then $t^N := \{ t^m : m \in N \}$
is definable in $A$.  We have a bijection $t^N \to {\mathbb F}_p[t]$ 
which sends $t^n$, where $n = \prod_{i \geq 1} p_i^{n_i}$,
to $\sum_{i \geq 0} n_i t^{i-1}$.  Rumely~\cite[p. 211]{Rum} established the 
definability of this map in $K$ (and hence in $A$).  In particular, 
${\mathbb F}_p[t]$ is definable in~$A$, and we have a definable 
injection ${\mathbb F}_p[t] \hookrightarrow t^{\mathbb N}$.  Since $A$ is 
a f.g.~free ${\mathbb F}_p[t]$-module, 
we also have an ${\mathbb F}_p[t]$-linear (hence definable) bijection 
$A \to {\mathbb F}_p[t]^d$, for some $d \geq 1$.  
The lemma now follows from Corollary~\ref{cor:Nies}. 
\end{proof}

\noindent
With our lemmata in place, we complete the proof of Theorem~\ref{thm:integral domains}.
Thus, suppose~$A$ is infinite, so $\dim(A)\geq 1$.

For each natural number $n$, Poonen~\cite{Poonen} produced 
a formula $\theta_n(x_1,\ldots,x_n)$ so that for any finitely generated 
field $F$ and any $n$-tuple $a = (a_1,\ldots,a_n) \in F^n$ one has 
$F \models \theta_n(a)$ if and only if the elements $a_1, \ldots, a_n$
are algebraically independent.   If $\operatorname{char}(A) = 0$, 
let $$D := \big\{ a \in A ~:~ K \models \neg \theta_1(a) \big\}.$$
If $\operatorname{char}(A) = p > 0$, then pick some $t \in A$ which is 
transcendental over ${\mathbb F}_p$ and set $$D := \big\{ a \in A ~:~ 
K \models \neg \theta_2(a, t) \big\}.$$  
In both cases, $D$ is an algebraically closed subring of $A$ with $\operatorname{adim}(D)=1$, definable in $A$.
By Lemma~\ref{onefg}, $D$ is finitely generated, hence noetherian, and therefore a Dedekind domain.

By Proposition~\ref{prop:NN} we take some nonzero $c \in D$  
and $x_1, \ldots, x_m \in A$ so that $x_1, \ldots, x_m$ are algebraically independent 
over $D$ and $A_c:=A[c^{-1}]$ is a finite integral extension of $D_c[x_1, \ldots, x_m]$,
where $D_c:=D[c^{-1}]$.   By Proposition~\ref{prop:generic freeness},
after further localizing at another nonzero element of $D$, we can also assume that $A_c$ is a free $D_c$-module.
Let $y_1, \ldots, y_n \in A$ be generators of
$A_c$ as $D_c[x_1, \ldots, x_m]$-module.   Let  $X=(X_1,\dots,X_m)$, $Y=(Y_1,\dots,Y_n)$ be tuples of indeterminates, and let
${\mathfrak p}$ be the kernel of the $D_c$-algebra morphism 
$D_c[X,Y] \to A_c$  given by 
$X_i \mapsto x_i$ and $Y_j \mapsto y_j$ for 
$i=1,\dots,m$ and $j=1,\dots,n$.  Note that $\mathfrak p\cap D_c[X]=(0)$.
Let $P_1, \ldots, P_\ell$ be a 
sequence of generators of ${\mathfrak p}$ 
and let 
$V = V({\mathfrak p}) \subseteq {\mathbb A}^{m+n}_{D_c}$ be 
the affine variety defined by ${\mathfrak p}$.  Then $A_c$ can be naturally 
seen as a subring of the ring of regular functions on~$V$. 
For any point $a \in {\mathbb A}^m(D_c)$ there is some 
integral domain $D'$ extending $D_c$, as a $D_c$-module generated by at most $n$ elements, and 
some point $b \in {\mathbb A}^n(D')$ so that $(a,b) \in V(D')$.

By Lemma~\ref{lem:interpret A[1/c]} we have an interpretation $D\leadsto D_c$. (We could have also used   
Lemma~\ref{biinterporder} or~\ref{biinterpposone}, in combination with
Examples~\ref{ex:ring interpretations},~(4) and Corollary~\ref{cor:biNdefpower}.)
Precomposing this interpretation with the interpretation $A\leadsto D$ given by the inclusion $D\subseteq A$
yields an interpretation of~$D_c$ in $A$.  Lemma~\ref{lem:interpret A[1/c]} also shows that~$A_c$ is interpretable in $A$.
Every ideal of a Dedekind domain (such as $D_c$) is generated by two elements. Hence by Lemma~\ref{lem:interpret finite extensions}
the class $\mathcal D$ of integral extensions of~$D_c$ generated by $n$ elements as $D_c$-modules is uniformly
interpretable in $D_c$ (and hence in $A$), and
by  Corollary~\ref{cor:interpret finite extensions}, the class of ring $A_c\otimes_{D_c} D'$ where $D'\in\mathcal D$ is  uniformly interpretable in $A_c$ (and hence in $A$).
As a consequence the following set is definable in~$A$:
\begin{eqnarray*}
E  :=  \big\{ (a, D', b, e, p) &:& a \in {\mathbb A}^m(D_c),\  
D' \in {\mathcal D}, \   b \in {\mathbb A}^n(D'), \\ &&
(a,b) \in V(D'),\ e \in D',\ p \in A, \text{ and } p(a,b) = e \big\} 
\end{eqnarray*}
Indeed, the 
condition that $(a,b) \in V(D')$ may be expressed by saying that 
$P_1(a,b) = \cdots = P_\ell(a,b) = 0$.  That $p(a,b) = e$ is expressed by 
saying $$(\exists u_1, \ldots, u_m, v_1, 
\ldots, v_n \in A_c \otimes_{D_c} D')
\left(p - e = \textstyle\sum_i v_i (x_i - a_i) + \sum_j u_j (y_j - b_j)\right).$$   
(To see this use that $A_c \otimes_{D_c} D'\cong D'[X,Y]/\mathfrak p D'[X,Y]$ as $D'$-algebras, and for all $(a,b)\in \mathbb A^{m+n}(D')$, the kernel of the morphism $p\mapsto p(a,b)\colon D'[X,Y]\to D'$ is generated by $X_i-a_i$ and $Y_j-b_j$.) We also note that given $p,q\in A$, we have
\begin{equation}\label{eq:p=q}
\qquad p=q \quad\Longleftrightarrow \quad \begin{cases}
&\parbox{23em}{$\big(\forall a \in \mathbb A^m(D_c)\big) (\forall D' \in \mathcal D)
\big(\forall b \in \mathbb A^n(D')\big)(\forall e \in D')$ \\[0.25em]
$(a, D', b, e, p)\in E \Leftrightarrow (a, D', b, e, q)\in E$.
}\end{cases}
\end{equation}
Let now $(f,g)$ be a bi-interpretation between $D_c$ and $\NN$, let $i$ be an interpretation of $D_c$ in $A$,
and let $h$ be an interpretation of $A$ in $\NN$ (Corollary~\ref{cor:fg ring int in ZZ}).
Put $j:=h\circ f\colon D_c\leadsto A$.  
Let $\widetilde{A}:=(j\circ i)^*(A)$ be the copy of~$A$
interpreted via $j$ in the copy of~$D_c$   interpreted via $i$ in $A$.
By Lemma~\ref{biinterporder} or~\ref{biinterpposone}, $D_c$ is bi-interpretable with $\NN$, so by Corollary~\ref{cor:self-interpretations},
the self-interpretation $i\circ j$ of $D_c$ is homotopic to the identity.
Thus if we can show that the isomorphism $\overline{j\circ i}\colon\widetilde{A}\to A$ is  definable in~$A$,
then the pair $(i,j)$ is a  bi-interpretation between $A$ and $D_c$.
Let $\alpha$ denote the inverse of~$\overline{j\circ i}$. Then $i\circ\alpha$ is an interpretation of $D_c$ in $\widetilde{A}$, and
by Corollary~\ref{cor:def iso},~$\alpha$ induces an isomorphism $i^*(D_c) \to (i\circ\alpha)^*(D_c)$
which is definable in $A$.
We also  denote this isomorphism by $\alpha$, and  also  denote by $\alpha$ the induced map
on the various objects defined in $i^*(D_c)$.
With this convention, put $\widetilde{E}:=\alpha(E)$. Then $\widetilde{E}$ is definable  in $\widetilde{A}$, and hence also in $A$. Therefore
\begin{eqnarray*}
\Gamma := \big\{ (p,\widetilde{p}) \in A \times \tilde{A}  &:& 
\big(\forall a \in \mathbb A^m(D_c)\big) (\forall D' \in \mathcal D)
\big(\forall b \in \mathbb A^n(D')\big)(\forall e \in D') \\ && (a,D',b,e,p) \in E 
\leftrightarrow \big(\alpha(a),\alpha(D'),\alpha(b),\alpha(e),\widetilde{p}\big) \in 
\widetilde{E} \big\}
\end{eqnarray*}
is definable in $A$, and by \eqref{eq:p=q}, $\Gamma$ is the graph of $\alpha \colon A\to\widetilde{A}$. This implies that $A$ is bi-interpretable with $D_c$, and hence with ${\mathbb N}$.    \qed

\section{Fiber Products}\label{sec:fiber products}

\noindent
In this section we study finitely generated rings which can be expressed as fiber products of other rings. 
We first review the definition, and then successively focus on fiber products over finite rings and fiber products over infinite rings. The section culminates with a characterization of those f.g.~reduced rings which are bi-interpretable with $\ZZ$.

\subsection{Definition and basic properties}
Let $\alpha\colon A\to C$ and $\beta\colon B\to C$ be two ring morphisms. 
The {\it fiber product} of $A$ and $B$ over $C$ is the subring
$$A \times_C B = \big\{ (a,b)\in A\times B: \alpha(a)=\beta(b) \big\}$$
of the direct product $A\times B$. The natural projections $A\times B\to A$ and $A\times B\to B$ restrict to
ring morphisms $\pi_A\colon A\times_C B\to A$ and $\pi_B\colon A\times_C B\to B$, respectively. Note that if $\alpha$ is surjective, then $\pi_B$ is surjective; similarly, if $\beta$ is surjective, then so is $\pi_A$.
{\it In the following we always assume that $\alpha$, $\beta$ are surjective.}\/ We do allow~$C$ to be the zero ring; in this case, $A\times_C B=A\times B$.

\begin{example}\label{ex:fiber product}
Let $I$, $J$ be ideals of a ring $R$. Then the natural morphism $R/(I\cap J)\to (R/I)\times (R/J)$ maps $R/(I\cap J)$ isomorphically onto the fiber product $A\times_C B$ of $A=R/I$ and $B=R/J$ over $C=R/(I+J)$, where $\alpha$ and $\beta$ are the natural morphisms $A=R/I\to C=R/(I+J)$ respectively $B=R/J\to C=R/(I+J)$.
\end{example}

\begin{lemma}
Suppose $A$ and $B$ are noetherian. Then $A\times_C B$ is noetherian.
\end{lemma}
\begin{proof}
Let $I=\ker\pi_A$, $J=\ker\pi_B$, and $R:=A\times_C B$. Since $I\cap J=0$, we  have a natural embedding of $R$ into the ring $(R/I) \times (R/J)$.
The ring morphism $\pi_A$, $\pi_B$ induce isomorphisms $R/I\to A$, $R/J\to B$. 
Thus $R/I$ and $R/J$ are noetherian as rings and hence as $R$-modules. So the product $(R/I) \times (R/J)$, and hence its submodule $R$, is a noetherian $R$-module as well.
\end{proof}

\begin{corollary}\label{cor:interpretation of the factors}
Suppose $A$ and $B$ are noetherian. Then $\pi_A$ is an interpretation of~$A$ in $A\times_C B$, and  $\pi_B$ is an interpretation of $B$ in $A\times_C B$, and hence $\pi_A\times\pi_B$ is an interpretation of $A\times B$ in $A\times_C B$.
\end{corollary}
\begin{proof}
By the previous lemma, the ideals $I=\ker\pi_A$ and $J=\ker\pi_B$ of $A\times_C B$ are f.g., and hence (existentially) definable in $A\times_C B$.
\end{proof}

\begin{lemma}\label{lem:Z-interpretation of fiber products}
Suppose $A$ and $B$ are interpretable in $\ZZ$ and $C$ is f.g. Then $A\times_C B$ is interpretable in $\ZZ$.
\end{lemma}
\begin{proof}
Let $f\colon \ZZ\leadsto A$ and $g\colon \ZZ\leadsto B$; then $f\times g$ is an interpretation $\ZZ\leadsto A\times B$. Both $\alpha\circ f$ and $\beta\circ g$ are interpretations $\ZZ\leadsto C$; so by Lemma~\ref{lem:claim 2} (and the assumption that $C$ is f.g.), the set 
$$[{\alpha\circ f=\beta\circ g}]=(f\times g)^{-1}(A\times_C B)$$ is definable in $\ZZ$.
Hence the restriction of $f\times g$ to a map $(f\times g)^{-1}(A\times_C B)\to A\times_C B$ is an interpretation of $A\times_C B$ in $\ZZ$.
\end{proof}

\subsection{Fiber products over finite rings}
Every fiber product of noetherian rings over a finite ring is bi-interpretable with the direct product of those rings:

\begin{lemma}
Let $\alpha\colon A\to C$ and $\beta\colon B\to C$ be  surjective morphisms of noetherian rings, where $C$ is finite. Then the pair~$(f,g)$, where~$f$ is the natural inclusion $A\times_C B\to A\times B$ and $g=\pi_A\times\pi_B$, is a bi-interpretation between $A\times B$ and~$A\times_C B$.
\end{lemma}
\begin{proof}
We first observe that the subset $M:=A\times_C B$ of $A\times B$ is definable in the ring $A\times B$ (and hence that $f$ is indeed an interpretation $A\times B\leadsto A\times_C B$). To see this first note that the map $\Pi_A\colon A\times B\to A\times 0$ given by
$(a,b)\mapsto (a,0)=(a,b)\cdot (1,0)$ is definable in $A\times B$ (with the parameter $(1,0)$); similarly, the
map $(a,b)\mapsto \Pi_B(a,b)=(0,b)\colon A\times B\to 0\times B$ is definable in $A\times B$.
Let $n=\abs{C}$ and let $a_1,\dots,a_n\in A$ be representatives for the residue classes of $A/\ker\alpha$
and $b_1,\dots,b_n\in B$ be representatives for the residue classes of $B/\ker\beta$ such that
$\alpha(a_i)=\beta_i(b_i)$ for $i=1,\dots,n$. Then $M$ is seen to be definable as the set of all $(a,b)\in A\times B$ such that for each $i\in\{1,\dots,n\}$,
$$(a,b)\in (a_i,0)+\Pi_A^{-1}(\ker\alpha)\quad\Longleftrightarrow\quad
(a,b)\in (0,b_i)+\Pi_B^{-1}(\ker\beta).$$ 
The self-interpretation $g\circ f$ of $A\times B$ is the map
$$\big((a,b),(a',b')\big)\mapsto \Pi_A(a,b)+\Pi_B(a',b')=(a,b')\colon M\times M\to A\times B$$ 
and hence definable in $A\times B$. 
Similarly, the self-interpretation $f\circ g$ of $A\times_C B$ is the map
$$\big((a,b),(a',b')\big)\mapsto (a,b')\colon g^{-1}(M) \to A\times_C B,$$
and since
$$(f\circ g)\big((a,b),(a',b')\big)=(a'',b'')\ \Longleftrightarrow\ (a'',b'')\in \big((a,b)+\ker\pi_A\big) \cap \big((a',b')+\ker\pi_B\big),$$ we also see
that $f\circ g\simeq \id_{A\times_C B}$.
\end{proof}

\noindent
The previous lemma leads us to the study of the bi-interpretability class of the direct product of two f.g.~rings.
We first observe that a product of a ring $B$ with a finite ring is (parametrically) bi-interpretable with $B$ itself:

\begin{lemma}\label{lem:finiteprodbi}
Let $A$ be a direct product $A=B\times R$ of a ring $B$ with a finite ring~$R$. Then $A$ and $B$ are bi-interpretable.
\end{lemma}
\begin{proof}
The surjective ring morphism 
$(b,r)\mapsto b\colon A\to B$ is an interpretation $f\colon A\leadsto B$ with $\ker f=A\cdot(0,1)$.
(See Example~\ref{ex:ring interpretations},~(2).)
Pick a bijection $g\colon R'\to R$ where $R'\subseteq  B^n$ for some $n>0$. Then the 
bijection 
$$(b,r')\mapsto \big(b,g(r')\big)\colon B\times R'\to B\times R=A,$$ 
in the following also denoted by $g$, is an interpretation $B\leadsto A$ (since the  addition and multiplication tables of the finite ring $R$ are  definable).
Now $f\circ g\colon B\times R'\to B$ is given by $(b,r')\mapsto b$ and hence definable in $B$, and
$$g\circ f\colon A\times f^{-1}(R') = f^{-1}(B\times R')\to A$$ is given by 
$$\big((b,r),(b_1,r_1),\dots,(b_m,r_m)\big)\mapsto \big(b,g(b_1,\dots,b_m)\big)$$
and thus definable in $A$, since $(b,r)\cdot (1,0) = (b,0)$ and $(b,r)\cdot (0,1) = (0,r)$ for all $b\in B$, $r\in R$.
This shows that $(f,g)$ is a bi-interpretation between $A$ and $B$.
\end{proof}

\noindent
On the other hand, the direct product of two infinite f.g.~rings is never bi-in\-ter\-pretable with $\ZZ$:

\begin{lemma} \label{lem:nbiprod}
Let $A$ and $B$ be infinite finitely generated rings. Then $A \times B$ is not bi-interpretable with $\ZZ$.
\end{lemma}
\begin{proof}
Let $a\in A$ and $b\in B$ be elements of infinite multiplicative order. (See Corol\-lary~\ref{cor:inf mult order}.) Suppose $A\times B$ is bi-interpretable with $\ZZ$. Then by Corollary~\ref{cor:biNdefpower}, the set~$(a,b)^\NN$ of powers of $(a,b)$ is definable in $A\times B$. By the Feferman-Vaught Theorem~\cite[Corollary~9.6.4]{Ho} there are  $N\in\NN$ and formulas $\varphi_i(x)$, $\psi_i(y)$ ($i=1,\dots,N$), possibly with parameters, such that for all $(a',b')\in A\times B$, we have
$$(a',b') \in (a,b)^\NN \quad\Longleftrightarrow\quad \text{$A\models\varphi_i(a')$ and  $B\models\psi_i(b')$,
for some $i\in\{1,\dots,N\}$.}$$
By the pigeon hole principle, there are $m\neq n$ and some $i\in \{1,\dots,N\}$ such that
$A\models \varphi_i(a^m)\wedge\varphi_i(a^n)$ and $B\models \psi_i(b^m)\wedge\psi_i(b^n)$. But then
$A\models \varphi_i(a^m)$ and $B\models \psi_i(b^n)$, so $(a^m,b^n)\in (a,b)^\NN$, a contradiction to $m\neq n$.
\end{proof}

\noindent
Combining the results in this subsection  immediately yields the following consequences:

\begin{corollary}
The fiber product of a noetherian ring $A$ with a finite ring is bi-in\-ter\-pretable with $A$.
\end{corollary}

\begin{corollary}\label{cor:fiber product over finite ring}
The fiber product of two infinite f.g.~rings over a finite ring is \emph{not} bi-interpretable with $\ZZ$.
\end{corollary}

\subsection{Fiber products over infinite rings}
In this subsection we show:

\begin{theorem}\label{thm:fiber product over infinite ring}
Let $\alpha\colon A\to C$ and $\beta\colon B\to C$ be  surjective ring morphisms.
If $A$ and $B$ are both bi-interpretable with~$\ZZ$, and $C$ is f.g.~and infinite, then~$A \times_C B$ is also bi-interpretable with $\ZZ$.
\end{theorem}

\noindent
For the proof, which is based on the criterion for bi-interpretability with $\ZZ$ from Corollary~\ref{cor:Nies}, we need:

\begin{lemma}\label{lem:Z-structure on powers}
Let $A$ be bi-interpretable with $\ZZ$, and $a\in A$  be of infinite multiplicative order. Then there exists a definable bijection $A\to a^\NN$, and hence definable binary operations $\oplus$ and $\otimes$ on $a^\NN$ making $a^\NN$ into a ring isomorphic to $\ZZ$.
\end{lemma}
\begin{proof}
Take an interpretation $f\colon\bA\leadsto\ZZ$ and a definable bijective map $\iota\colon A\to f^*(\ZZ)$. (See the beginning of Section~\ref{sec:Nies}.) Choose a definable bijection $f^*(\ZZ)\to f^*(\NN)$. 
By Corollary~\ref{cor:power map}, the map $\overline{n}\mapsto a^n\colon f^*(\NN)\to a^\NN$ is definable.
Thus the composition
$$A\xrightarrow{\ \iota\ } f^*(\ZZ)\to f^*(\NN)\xrightarrow{\ \overline{n}\mapsto a^n\ } a^\NN$$
is a definable bijection as required. The rest follows from Lemma~\ref{lem:oplus and otimes}.
\end{proof}

\noindent
We also use the following number-theoretic fact:

\begin{theorem}[{Scott \cite[Theorem~3]{Scott}}] \label{thm:Scott}
Let $p$, $q$ be distinct prime numbers and $c\in\ZZ$. Then there is at most one pair $(m,n)$ with $p^{2m}-q^{2n}=c$.
\end{theorem}

\noindent
We  now show  Theorem~\ref{thm:fiber product over infinite ring}.
Thus, assume that $A$ and $B$ are bi-interpretable with~$\ZZ$, and $C$ is f.g.~and infinite.
By Lemma~\ref{lem:Z-interpretation of fiber products},  $R := A \times_C B$ is interpretable in $\ZZ$, so by Corollary~\ref{cor:Nies}, in order to see that $R$ is bi-interpretable with $\ZZ$, it is enough to show that we can interpret $\ZZ$ in the ring $R$ such that~$R$ can be mapped definably and injectively into the interpreted copy $Z$ of $\ZZ$ in $R$.

To see this, let $a \in A$ and $b \in B$ so that $\alpha(a) = \beta(b)$ has infinite multiplicative order in~$C$. 
Then $Z:=(a,b)^\NN$ is definable in $R$ as 
$$Z = \big\{ r \in R :  \text{$\pi_A(r) \in a^\NN$ and $\pi_B(r) \in b^\NN$} \big\}.$$  
(Clearly, $Z$ is contained in the set on the right-hand side of this equation; conversely, if $r$ is any element of this set, then $r = (a^m,b^n)$ for some $m$ and $n$, with $\alpha(a^m) = \beta(b^n)$, and then $\alpha(a)^m = \alpha(a)^n$,
as $\alpha(a) = \beta(b)$,   forcing $m=n$ since $\alpha(a)$ has infinite multiplicative order.)

Recall from Corollary~\ref{cor:interpretation of the factors} that $\pi_A$ is an interpretation $R\leadsto A$.
We denote by~$\overline{A}:=\pi_A^*(A)=R/\ker\pi_A$ the copy of $A$ in $R$ interpreted via $\pi_A$, 
and by~${x\mapsto\overline{x}\colon A\to\overline{A}}$ the natural isomorphism; similarly with $B$ in place of $A$.
The natural surjection~$R\to\overline{A}$ restricts to a bijection $Z=(a,b)^\NN\to \overline{a}^\NN$; we denote by $e_A$ its inverse, and we define~$e_B$ similarly. Note that $e_A$ and $e_B$ are definable in~$R$.
By Lemma~\ref{lem:Z-structure on powers} there are binary operations on $\overline{a}^\NN$, definable in $\overline{A}$, which make $\overline{a}^\NN$ into a ring isomorphic to $(\ZZ,{+},{\times})$. Equip $Z$ with binary operations $\oplus$, $\otimes$ making $e_A$ a ring isomorphism; then $\oplus$, $\otimes$ are definable in~$R$, and $(Z,{\oplus},{\otimes})\cong (\ZZ,{+},{\times})$.

It remains to specify a definable injective map $R\to Z$.
Let $f_A\colon\overline{A}\to\overline{a}^\NN$ and $f_B\colon\overline{B}\to\overline{b}^\NN$ be definable bijections, according to Lemma~\ref{lem:Z-structure on powers}, and let $F_A$ and $F_B$ be the composition of $f_A$, $f_B$ with the natural surjection $R\to\overline{A}$ and $R\to\overline{B}$, respectively; then $F_A$, $F_B$ are definable in $R$.
From Corollary~\ref{cor:power map} and the fact that exponentiation is definable in $\NN$, we see that the maps
$t_A\colon \overline{a}^\NN\to\overline{a}^\NN$ and $t_B\colon \overline{b}^\NN\to\overline{b}^\NN$ given by
$t_A(\overline{a}^m)=\overline{a}^{2^{2m}}$ and $t_B(\overline{b}^n)=\overline{b}^{3^{2n}}$ are definable.
It is now easy to verify, using Theorem~\ref{thm:Scott}, that the definable map
$$r\mapsto (e_A\circ t_A\circ F_A) (r) \cdot (e_B\circ t_B\circ F_B) (r)\colon R\to Z$$
is injective. \qed

\begin{remark*}
Below we apply Theorem~\ref{thm:fiber product over infinite ring} in a situation where we know a priori that the ring~$A\times_C B$ is f.g. We do not know whether the fiber product of two f.g.~rings is always again f.g.
\end{remark*}

\subsection{The graph of minimal non-maximal prime ideals}
Let $A$ be a ring.
We denote by $\Min(A)$ the set of minimal prime ideals of $A$; we always assume that $\Min(A)$ is finite. (This is the case if $A$ is noetherian.)
We define a (simple, undirected) graph $\mathcal G_A=(V,E)$ whose vertex set is the set $V=\Min(A)\setminus\Max(A)$ 
of all minimal non-maximal prime ideals of $A$, and whose edge relation is defined by
$$(\mathfrak p,\mathfrak q)\in E \quad :\Longleftrightarrow\quad \text{there is a non-maximal prime ideal of $A$ containing $\mathfrak p+\mathfrak q$.}$$ 
Note that if $A$ is f.g., then $V$ is the set of minimal prime ideals of $A$ of infinite index in $A$
(so $V\neq\emptyset$ iff $A$ is infinite), and $(\mathfrak p,\mathfrak q)\in E$ iff $\mathfrak p+\mathfrak q$ is of infinite index in~$A$. (See Corollary~\ref{cor:finite = 0-dim}.)

\medskip
\noindent
We first relate connectedness of the graph $\mathcal G_A$ with connectedness of the topological space $\Spec^\circ(A)=\Spec(A)\setminus\Max(A)$ considered in the introduction.

\begin{lemma} \label{lem:dim 0 for sum of intersections}
Let $I_1,\dots,I_m,J_1,\dots,J_n$ be ideals of $A$, where $m,n\geq 1$, and $I=I_1\cap\cdots\cap I_m$, $J=J_1\cap\cdots\cap J_n$.
Suppose that every  prime ideal containing $I_i+J_j$ is maximal, for all $i=1,\dots,m$ and $j=1,\dots,n$. Then every  prime ideal containing $I+J$ is maximal. 
\end{lemma}
\begin{proof}
Let $\mathfrak p\supseteq I+J$ be a prime ideal. 
If a prime ideal contains an intersection of finitely many ideals, then it contains one of them. Hence 
there are $i$ and $j$ such that $\mathfrak p\supseteq I_i+J_j$, and thus $\mathfrak p$ is maximal.
\end{proof}

\noindent
Given an ideal $I$ of $A$ we let $V(I)$ be the closed subset of $\Spec(A)$ consisting of all $\mathfrak p\in\Spec(A)$ containing $I$. 

\begin{cor}
$\mathcal G_A$ is connected iff $\Spec^\circ(A)$ is connected.
\end{cor}
\begin{proof}
Suppose first that $\Spec^\circ(A)$ is disconnected, that is,
there are nonempty closed subsets $X$, $X'$ partitioning $\Spec^\circ(A)$. Then both $X$ and $X'$ contain
a non-maximal minimal prime ideal. Let $C$ be the set of non-maximal minimal prime ideals contained in $X$, and $C':=V\setminus C$.
For $\mathfrak p\in C$ and $\mathfrak p'\in C'$, we have $$\Spec^\circ(A)\cap V(\mathfrak p+\mathfrak p')=\Spec^\circ(A)\cap V(\mathfrak p)\cap V(\mathfrak p')\subseteq X\cap X'=\emptyset$$
and thus $(\mathfrak p,\mathfrak p')\notin E$. Hence $\mathcal G_A$ is disconnected.

Conversely, suppose $\mathcal G_A$ is disconnected. Let $C$, $C'$ be  nonempty sets partitioning~$V$ such that 
$(\mathfrak p,\mathfrak p')\notin E$ for all $\mathfrak p\in C$, $\mathfrak p'\in C'$.
Put $I:=\bigcap C$, $I':=\bigcap C'$. Then $X:=V(I)\cap \Spec^\circ(A)$, $X':=V(I')\cap \Spec^\circ(A)$ are nonempty closed subsets of $\Spec^\circ(A)$ with $X\cup X'=\Spec^\circ(A)$, and by the previous lemma we have $X\cap X'=\emptyset$. Thus $ \Spec^\circ(A)$ is disconnected.
\end{proof}

\begin{remark*}
In the case where $A$ is a local ring, the graph $\mathcal G_A$ has been considered in different contexts. (See, e.g., \cite[Definition~3.4]{HoHu} or \cite[Remark~2.3]{SiWa}.)
\end{remark*}

\noindent
The following lemma allows us to analyze the graph $\mathcal G_A$ by splitting off a single vertex:

\begin{lemma}\label{lem:A0, 1}
Let  $\mathfrak p_0\in \Min(A)$,
$$I_0:=\bigcap\big\{\mathfrak p\in\Min(A): \mathfrak p\neq\mathfrak p_0\big\},$$ 
and $A_0:=A/I_0$, with natural surjection $a\mapsto\overline{a}=a+I_0\colon A\to A_0$. Then 
$$\mathfrak p\mapsto \overline{\mathfrak p}\colon\Min(A)\setminus\{\mathfrak p_0\}\to\Min(A_0)$$
is a bijection. Moreover, for $\mathfrak p,\mathfrak q\in\Min(A)\setminus\{\mathfrak p_0\}$
the natural surjection $A\to A_0$ induces an isomorphism $A/(\mathfrak p+\mathfrak q)\to A_0/(\overline{\mathfrak p}+\overline{\mathfrak q})$.
\end{lemma}
\begin{proof}
The map $\mathfrak p\mapsto\overline{\mathfrak p}$ is an inclusion-preserving 
correspondence between the set~$V(I_0)$ of prime ideals of $A$ containing $I_0$ and the set of all prime ideals of~$A_0$.
Clearly $\Min(A)\setminus\{\mathfrak p_0\}\subseteq V(I_0)$, and
if $\mathfrak p\supseteq I_0$ is a minimal prime ideal of $A$, then~$\overline{\mathfrak p}$ is a minimal prime ideal of $A_0$. To show surjectivity, let $\overline{\mathfrak q}$ be a minimal prime ideal of $A_0$, where $\mathfrak q\in V(I_0)$. 
Then $\mathfrak q\supseteq\mathfrak p\supseteq I_0$ for some $\mathfrak p\in\Min(A)$ with $\mathfrak p\neq\mathfrak p_0$, and so $\mathfrak q=\mathfrak p$ by minimality of $\overline{\mathfrak q}$.
The rest of the lemma is easy to see.
\end{proof}

\noindent
Given a graph $\mathcal G=(V,E)$ and a vertex $v\in V$, we denote by $\mathcal G\setminus v$ the graph obtained from $\mathcal G$ by removing $v$, i.e., the graph with vertex set $W=V\setminus\{v\}$ and edge set $E\cap (W\times W)$.
If $\mathfrak p_0$ is a minimal non-maximal prime of $A$ and $I_0$ and $A_0$ are as in Lemma~\ref{lem:A0, 1}, then
$\mathfrak p\mapsto\overline{\mathfrak p}$ is an isomorphism $\mathcal G_A \setminus \mathfrak p_0 \to \mathcal G_{A_0}$.

\medskip
\noindent
We now return to bi-interpretability issues:

\begin{lemma}\label{lem:connected comp}
Suppose $A$ is infinite and f.g. Let $C\subseteq V$, $C\neq\emptyset$, such that the induced subgraph $\mathcal G_A{\upharpoonright} C$ of $\mathcal G_A$ with vertex set $C$ is connected, and let $I=\bigcap C$. Then $A/I$ is bi-interpretable with $\ZZ$.
\end{lemma}
\begin{proof}
We proceed by induction on the size of $C$. If $\abs{C}=1$, then $I$ is a prime ideal of infinite index, and the claim holds by Theorem~\ref{thm:integral domains}.
So suppose $\abs{C}>1$. It is well-known that each non-trivial finite connected graph $\mathcal G$ contains a non-cut vertex, i.e., a vertex $v$ such that $\mathcal G\setminus v$ is still connected.
Thus, let $\mathfrak p_0$ be a non-cut vertex of $\mathcal G_A {\upharpoonright} C$, and let
$C_0:=C\setminus\{\mathfrak p_0\}$, $I_0:=\bigcap C_0$. 
Choosing $\mathfrak p\in C_0$ such that $(\mathfrak p,\mathfrak p_0)\in E$, we have
$I_0+\mathfrak p_0\subseteq \mathfrak p+\mathfrak p_0$ and $A/(\mathfrak p+\mathfrak p_0)$ is infinite; hence $A/(I_0+\mathfrak p_0)$ is infinite.
By Example~\ref{ex:fiber product}, the rings $A/I=A/(I_0\cap\mathfrak p_0)$ and $(A/I_0) \times_{A/(I_0+\mathfrak p_0)} (A/\mathfrak p_0)$ are naturally isomorphic, where
$A/I_0$ and $A/\mathfrak p_0$ are both bi-interpretable with $\ZZ$,
by inductive assumption and Theorem~\ref{thm:integral domains}, respectively.
Hence $A/I$ is bi-interpretable with $\ZZ$ by Theorem~\ref{thm:fiber product over infinite ring}.
\end{proof}

\noindent
For the next lemma note that the graphs $\mathcal G_A$ and $\mathcal G_{A_{\operatorname{red}}}$  are naturally isomorphic.

\begin{lemma}\label{lem:A0, 2}
Let $\mathfrak p_0\in\Min(A)$ be of finite index in~$A$, and let $I_0$ and $A_0$ be as in
Lemma~\ref{lem:A0, 1}.
Then the reduced rings $A_{\operatorname{red}}=A/N(A)$ and $A_0$ are bi-interpretable, and
the graphs $\mathcal G_A$ and $\mathcal G_{A_0}$ are naturally isomorphic.
\end{lemma}
\begin{proof}
We may assume that $I_0\not\subseteq\mathfrak p_0$ (since otherwise $I_0=N(A)$ and so $A=A_0$). Then $A=I_0+\mathfrak p_0$, since $\mathfrak p_0$ is a maximal ideal of $A$ (every finite integral domain is a field).
So the natural morphism $A\to A_0\times R$, where $R=A/\mathfrak p_0$, is surjective (by the Chinese Remainder Theorem) with kernel $N(A)=\bigcap\Min(A)=I_0\cap\mathfrak p_0$. The first claim now follows from Lemma~\ref{lem:finiteprodbi}.
For the second claim note that the prime ideals of $A_0\times R$ are the ideals of this ring having the form $\mathfrak p\times R$ where $\mathfrak p\in\Spec(A_0)$ or $A_0\times\mathfrak q$ where $\mathfrak q\in \Spec(R)$, and the latter all have finite index.
\end{proof}

\subsection{Characterizing the reduced rings which are bi-interpretable with~$\ZZ$}
Combining the results obtained so far in this section, we obtain the following characterization of those finitely generated reduced rings which are (parametrically) bi-interpretable with $\ZZ$.

\begin{theorem} \label{thm:biNreduced}
Let $A$ be an infinite finitely generated reduced ring. Then $A$ is bi-in\-ter\-pretable with~$\ZZ$ if and only if 
the graph $\mathcal G_A$ is connected.
\end{theorem}
\begin{proof}
After applying lemmata~\ref{lem:A0, 1} and \ref{lem:A0, 2} sufficiently often, we can reduce to the situation that no minimal prime of $A$ is maximal, i.e., the
vertex set of the graph~$\mathcal G_A$ equals $\Min(A)$.
In this case, if $\mathcal G_A$ is connected, then by Lemma~\ref{lem:connected comp}, the ring~$A$ is bi-interpretable with~$\ZZ$.
Conversely, suppose that $\mathcal G_A$ is not connected.
Let~${C\subseteq V}$ be a connected component of the graph $\mathcal G_A=(V,E)$.
Then for each~${\mathfrak p\in C}$ and~${\mathfrak q\in V\setminus C}$ we have $(\mathfrak p,\mathfrak q)\notin E$, i.e.,
$\mathfrak p+\mathfrak q$ has finite index in $A$.
Thus by Corollary~\ref{cor:finite = 0-dim} and
Lemma~\ref{lem:dim 0 for sum of intersections},  setting $I:=\bigcap C$, $J:=\bigcap (V\setminus C)$, the ideal~${I+J}$ has finite index in~$A$.
Since $I\cap J=N(A)=0$, by Example~\ref{ex:fiber product}, the rings $A$ and $(A/I) \times_{A/(I+J)} (A/J)$ are naturally isomorphic, and by Lemma~\ref{lem:connected comp}, both~$A/I$ and~$A/J$ are infinite.
Hence by Corollary~\ref{cor:fiber product over finite ring}, $A$ is not bi-interpretable with $\ZZ$.
\end{proof}

\section{Finite Nilpotent Extensions}\label{sec:finite nilpotent ext}

\noindent
Throughout this section we let $B$ be a ring with nilradical $N$. Our main goal for this section is the proof of the following theorem:

\begin{theorem}\label{thm:nilpotent}
Suppose $B$ is f.g.~and $\ann_\ZZ(N)\neq 0$. Then the rings $B_{\operatorname{red}}=B/N$ and $B$ are bi-in\-ter\-pret\-a\-ble.
\end{theorem}

\noindent
In particular, if $B$ is f.g.~and has positive characteristic, then $B_{\operatorname{red}}$ and $B$ are bi-in\-ter\-pret\-a\-ble.
Our bi-interpretation between $B_{\operatorname{red}}$ and $B$ passes through a truncation of Cartier's ring of big  Witt vectors over $B_{\operatorname{red}}$; therefore we first briefly review this construction. (See \cite[IX, \S{}1]{Bou} or \cite[\S{}17]{Haz} for missing proofs of the statements in the next subsection.)

\subsection{Witt vectors}
In the rest of this section we let  $d,i,j\geq 1$ be integers.
Let $X_1,X_2,\dots$ be countably many pairwise distinct 
indeterminates, and for each $j$ set $X_{|j}:=(X_i)_{i|j}$.
The $j$-th {\it Witt polynomials}\/ $w_j\in\ZZ[X_{|j}]$ are defined by
$$w_j := \sum_{i|j}  i X_i^{j/i} .$$
Let now  $Y_1,Y_2,\dots$ be another sequence of pairwise distinct indeterminates.
Then for any polynomial $P\in\ZZ[X,Y]$ in distinct indeterminates $X$, $Y$ there is a sequence~$(P_i)$ of polynomials $P_i\in\ZZ[X_{|i},Y_{|i}]$
such that 
$$P\big(w_i(X_{|i}), w_i(Y_{|i})\big) = w_i\big(P_1(X_1,Y_1),\dots, P_i(X_{|i}, Y_{|i})\big)
\quad\text{for all~$i$.}$$
In particular, there are sequences $(S_i)$ and $(M_i)$ of polynomials $S_i\in \ZZ[X_{|i},Y_{|i}]$
and $M_i\in \ZZ[X_{|i},Y_{|i}]$ such that 
\begin{align*}
w_i(X_{|i})+w_i(Y_{|i}) &= w_i\big(S_1(X_1,Y_1),\dots,S_i(X_{|i}, Y_{|i})\big), \\
w_i(X_{|i})\cdot w_i(Y_{|i}) &= w_i\big(M_1(X_1,Y_1),\dots,M_i(X_{|i}, Y_{|i})\big) 
\end{align*}
for all $i$. For example, $S_1=X_1+Y_1$, $M_1=X_1\cdot Y_1$, and if $p$ is a prime, then
$$S_p 	= X_p+Y_p-\sum_{i=1}^{p-1} \frac{1}{p} {p\choose i}  X_1^i Y_1^{p-i}, \qquad	 M_p	= X_1^p Y_p + X_p Y_1^p + pX_pY_p.$$
Let $A$ be a ring. We
let $A^{|d}$ be the set of sequences $a=(a_i)$ indexed by all $i|d$, and for $a=(a_i)\in A^{|d}$ and $j|d$ let
$a_{|j}:=(a_i)_{i|j}\in A^{|j}$.
We
define binary operations~$+$ and~$\cdot$ on $A^{|d}$ by
\begin{align*}
a+b 			&:= \big(S_1(a_1,b_1),\dots,S_j(a_{|j},b_{|j}),\dots\big), \\
a\cdot b		&:= \big(M_1(a_1,b_1),\dots,M_j(a_{|j},b_{|j}),\dots\big)
\end{align*}
for $a=(a_i),b=(b_i)\in A^{|d}$. 
Equipped with these operations, $A^{|d}$ becomes a ring (with~$0$ and $1$ given by $(0,0,0,\dots)$ and  $(1,0,0,\dots)$ respectively), which we call the $d$-th \emph{ring of Witt vectors}\/ over $A$, denoted by $W_{d}(A)$. Every ring morphism~${f\colon A\to B}$ induces a componentwise map $A^{|d}\to B^{|d}$,
and this map is a ring morphism 
$W_{d}(f)\colon W_d(A)\to W_d(B)$.
Thus~$W_d$ is a functor from the category of rings to itself. The polynomials $w_j$ define (functorial) ring morphisms 
$$a\mapsto w_j(a_{|j})\colon W_d(A)\to A,$$
and hence give rise to a ring morphism 
$$a\mapsto w_*(a):=\big(w_j(a_{|j})\big)\colon W_d\to  A^{|d},$$ 
where~$A^{|d}$ carries the product ring structure.
If no $i|d$ is a zero-divisor in $A$, then~$w_*$ is injective, and if
all $i|d$  are units in $A$, then $w_*$ is bijective.
Note that the underlying set of
$W_d(A)=A^{|d}$ is simply a finite-fold power of $A$, 
and   $A^{|d}$ is integral over the image of $W_d(A)$ under $w_*$. Moreover:

\begin{lemma}\label{lem:witt noeth}
If $A$ is f.g, then so is $W_d(A)$.
\end{lemma}
\begin{proof}
Using that $A$ is the image of a polynomial ring over $\ZZ$, we first reduce to the case
that $\operatorname{char}(A)=0$, so $w_*$ is injective.  
Since $A^{|d}$ is integral over $B:=w_*\big(W_d(A)\big)$, if the ring $A$ is f.g., then so is  
$B$, by the Artin-Tate Lemma~\ref{lem:Artin-Tate}, and hence so is~$W_d(A)$.
\end{proof}

\noindent
Note also that the identity map $A^{|d}\to W_d(A)$ furnishes us with an interpretation~$A\leadsto W_d(A)$ of the ring 
$W_d(A)$ in the ring $A$.

\subsection{A bi-interpretation between $B$ and $B_{\operatorname{red}}$}
Let  $I$ be  an ideal of $B$ with $I^2=0$
and $d\geq 1$ an integer such that $dI=0$. Put $A:=B/I$.
The residue morphism $B\to A$ induces a surjective ring morphism $r\colon W_d(B)\to W_d(A)$,
and we also have a ring morphism 
$$b=(b_i)\mapsto w(b):=w_d(b)=\sum_{i|d} ib_i^{d/i}\colon W_d(B)\to B.$$
The morphism $w$ descends to $W_d(A)$:

\begin{lemma}
There is a unique ring morphism $t\colon W_d(A)\to B$ such that $w  = t\circ r$.
\end{lemma}
\begin{proof}
Let $b=(b_i)$ and $b'=(b_i')$ be elements of $W_d(B)$   such that $r(b)=r(b')$, that is, $x_i:=b_i'-b_i\in I$
for each $i|d$.
Then for $i|d$ we have
$$ i (b_i')^{d/i} = i b_i^{d/i} + i (d/i) b_i^{ d/i - 1} x_i + \text{multiples of $x_i^2$,}$$
and since $x_i^2 = 0$ and $i (d/i) x_i = d x_i = 0$, we obtain $i (b_i')^{d/i} = i b_i^{d/i}$.
This yields $w_d(b)=w_d(b')$. So given $a\in W_d(A)$ we can set $t(a):=w_d(b)$ where $b$ is any
element of $W_d(B)$ with $r(b)=a$. One verifies easily that then $t\colon W_d(A)\to B$ has the required property. The uniqueness part is clear.
\end{proof}

\noindent
In the following we  
view $B$ as a $W_d(A)$-module via the morphism  $t$  from the previous lemma.

\begin{lemma}
Suppose $B$ is f.g.~Then the $W_d(A)$-module $B$ is f.g.
\end{lemma}
\begin{proof}
First note that the image $W$ of $W_d(A)$ under $t$ contains all $d$-th powers of elements of $B$.
Hence $B$ is integral over its subring $W$: each $b\in B$ is a zero of the monic polynomial $X^d-b^d$
with coefficients in $W$. Since $B$ is a f.g.~$W$-algebra, this implies that $B$ is a f.g.~$W_d(A)$-module \cite[Corollary~5.2]{AM}.
\end{proof}

\noindent
{\it In the rest of this subsection we assume that $B$ is f.g.}\/ Let $b_1,\dots,b_m$ be generators for the $W_d(A)$-module $B$, and consider the surjective $W_d(A)$-bilinear map
\begin{equation}\label{eq:interpret in Witt vectors}
(a_1,\dots,a_m) \mapsto \sum_{j=1}^m a_jb_j\colon W_d(A)^m \to B.
\end{equation}
By Lemma~\ref{lem:witt noeth}, the ring $W_d(A)$ is f.g., hence noetherian, so the kernel
of \eqref{eq:interpret in Witt vectors} is~f.g.~Using $W_d(A)$-bilinearity, the preimage of the graph
of multiplication in $B$ under the map  \eqref{eq:interpret in Witt vectors} is definable in $W_d(A)$.  
Hence the map  \eqref{eq:interpret in Witt vectors} is an interpretation of $B$ in $W_d(A)$.
Composing this interpretation $W_d(A)\leadsto B$ with the interpretation $A\leadsto W_d(A)$
from the previous subsection, we obtain an interpretation $f\colon A\leadsto B$. Since the ideal $I$ is f.g., the residue morphism $b\mapsto\overline{b}\colon B\to A=B/I$ is an interpretation $g\colon B\leadsto A$. With these notations, we have:

\begin{lemma}
The pair $(f,g)$ is a bi-interpretation between $A$ and $B$.
\end{lemma}
\begin{proof}
The self-interpretation $f\circ g$ of $B$ is the map $(B^{|d})^m\to B$ given by
$$(\beta_1,\dots,\beta_m)\mapsto \sum_{j=1}^m w_d(\beta_j) b_j$$
and hence definable in $B$. One also checks easily that the self-interpretation $g\circ f$ of $A$
is the map $(A^{|d})^m\to A$ given by
$$(\alpha_1,\dots,\alpha_m)\mapsto \sum_{j=1}^m w_d(\alpha_j) \overline{b_j},$$
hence definable in $A$.
\end{proof}

\noindent
We can now prove the main result of this section:

\begin{proof}[Proof of Theorem~\ref{thm:nilpotent}]
Since $\ann_\ZZ(N)\neq 0$, we can take some $d\geq 1$ with~${dN=0}$.
Since $B$ is f.g.~and hence noetherian, we can take some $e\in\NN$ with $N^{2^e}=0$.
We proceed by induction on $e$ to show that $B$ and $B_{\operatorname{red}}=B/N$ are
bi-interpretable. If $e=0$ then $N=0$, and there is nothing to show, so suppose~$e\geq 1$.
By the above applied to the ideal $I:=N^{2^{e-1}}$ of $B$ (so $I^2=0$), 
the f.g.~rings $A:=B/I$ and~$B$ are bi-interpretable.
Now the nilradical of $A$ is $N(A)=N+I$, so $dN(A)=0$ and $N(A)^{2^{e-1}}=0$. Hence by inductive hypothesis applied to $A$ in place of $B$, the rings $A_{\operatorname{red}}=A/N(A)$ and~$A$ are bi-interpretable. 
Since $A_{\operatorname{red}}$ and $B_{\operatorname{red}}$ are isomorphic and the relation of
bi-interpretability is transitive, this implies that $B_{\operatorname{red}}$ and $B$ are bi-interpretable.
\end{proof}

\section{Derivations on Nonstandard Models}\label{sec:derivations}

\noindent
In this section we shall construct derivations on nonstandard models of finitely generated rings.  Our appeal to ultralimits is not strictly speaking necessary as a simple compactness argument would suffice, but the systematic use of ultralimits permits us to avoid some syntactical considerations.

\subsection{Ultralimits}
Let us recall some of the basic formalism of ultralimits.  Let $I$ be a nonempty index set, $\cU$ be an ultrafilter on $I$, and $\bM=(M,\dots)$ be a structure (in some first-order language).  We denote by $\bM^\cU$ the ultrapower $\bM^I/\cU$ of~$\bM$ relative to $\cU$ and by $\Delta_{\bM}$ the diagonal embedding of $\bM$ into $\bM^\cU$, that is, that is, the embedding $\bM \to \bM^\cU$ induced by the map $M \to M^I$ which associates to an element $a$ of $M$ the constant function $I\to M$ with value $a$.   We define the ordinal-indexed directed system of ultralimits $\Ult_\cU(\bM,\alpha)$ by 
\begin{enumerate}
\item $\Ult_\cU(\bM,0) := \bM$, 
\item $\Ult_\cU(\bM,\alpha + 1) := \big(\!\Ult_\cU(\bM,\alpha)\big)^\cU$, 
and 
\item $\Ult_\cU(\bM,\lambda) := \varinjlim_{\alpha < \lambda} \Ult_\cU(\bM,\alpha)$ for a limit ordinal $\lambda$.  
\end{enumerate}
For us, in (3) only the case of $\lambda = \omega$ is relevant.  By way of notation, if $I$ and~$\cU$ are understood, then by an ultralimit we mean $\Ult_\cU(\bM,\omega)$ and we shall write $${^*}\bM := \Ult_\cU(\bM,\omega).$$
By definition of the direct limit,  the structure ${^*}\bM$ comes with a family of embeddings $\Ult_\cU(\bM,n)\to {^*}\bM$ which commute with the diagonal embeddings 
$$\Delta_{\Ult(\bM,n)}\colon \Ult_\cU(\bM,n) \to \Ult_\cU(\bM,n+1).$$ 
We identify $\bM$ with its image in ${^*}\bM$ under the embedding 
$$\bM=\Ult_\cU(\bM,0)\to {^*}\bM.$$
For fixed $\cU$, the ultralimit construction commutes with taking reducts, and is functorial on the category of sets.  Given a set $N$ and a map $f\colon M \to N$, we write ${^*}f\colon {^*}M \to {^*}N$ for the ultralimit of $f$. In particular, if $\bN$ is a substructure of $\bM$, then the ultralimit of the natural inclusion $N\to M$ is an embedding ${^*}\bN\to{^*}\bM$ (compatible with the inclusions of $\bN$ and $\bM$ into their respective ultralimits), by which we identify ${^*}\bN$ with a substructure of ${^*}\bM$.   From the universal property of the direct limit, we have the curious and useful fact that ${^*}\!\Ult_\cU(\bM,1) = {^*}\bM$ where by equality we mean canonical isomorphism. 

\subsection{Constructing derivations on elementary extensions}
With the next two lemmata we show that every non-principal ultrapower of an integral domain of characteristic zero admits an ultralimit carrying a derivation which is nontrivial on the ultralimit of the non-standard integers. 
We let $R$ be an integral domain of characteristic zero and $k\subseteq R$ be a subring.

\begin{lemma}
\label{fgderiv}
Suppose $R$ is a f.g.~$k$-algebra, and let $t \in R$ be transcendental over~$k$.  Then there is a $k$-derivation $\partial\colon R \to R$ with $\partial(t) \neq 0$.
\end{lemma}
\begin{proof}
Present $R$ as $R = k[t_1,\ldots,t_n]$ where $t_1 = t$.   Since $t$ is transcendental over~$k$, there is a $k$-derivation $D\colon K \to K$ on the field of fractions $K$ of $R$ satisfying $D(t) = 1$.  
(See, e.g., \cite[Proposition~VIII.5.2]{Lang}.)
Write $D(t_i) = a_i/b_i$ where $a_i \in R$ and $b_i \in R$, $b_i\neq 0$.  Let $\partial$ be the restriction of $\left(\prod_{i=1}^n b_i\right) D$ to $R$,  a $k$-derivation on~$R$ possibly taking values in $K$.  Since $R$ is an integral domain, $\partial(t) = \prod_i b_i \neq 0$, and visibly $\partial(t_i) = a_i \prod_{j \neq i} b_j \in R$ for each $i = 1,\dots,n$.  Hence, for any $f \in R$, writing $f = F(t_1,\ldots,t_n)$ for some polynomial $F$ over $k$, we see that $\partial(f) = \sum_{i=1}^n \frac{\partial F}{\partial X_i}(t_1,\ldots,t_n) \partial(t_i) \in R$.
\end{proof}

\noindent
Taking an ultralimit of the above derivations, we find interesting derivations on ultralimits.

\begin{lemma}
\label{eeder}
Let $t \in R$ be transcendental over $k$. Then there is an ultralimit ${^*}R$ of $R$ and a $k$-derivation $\partial\colon {^*}R \to {^*}R$ with $\partial(t) \neq 0$.
\end{lemma}
\begin{proof}
Let  $I$ be the set of finite subsets of $R$. For $S \in I$, let 
$$(S) := \{ S' \in I : S \subseteq S' \},$$
and let 
$\cC := \big\{ (S) : S \in  I \big\}$.
Observe that $\cC$ has the finite intersection property: $(S_1) \cap (S_2) = (  S_1 \cup S_2 )$ for all $S_1,S_2\in I$. Hence $\cC$ extends to an ultrafilter $\cU$ on $I$.

For each $S \in I$,   by Lemma~\ref{fgderiv} we may find a $k$-derivation $\partial_S\colon k[t,S] \to k[t,S]$ with $\partial_S(t) \neq 0$, and these $k$-derivations combine to a $k$-derivation $\prod_{S \in I} \partial_S$ on
the $k$-subalgebra $\prod_{S\in I} k[t,S]$ of $R^I$, which in turn induces a $k$-derivation $\partial_{\operatorname{fin}}$
with $\partial_{\operatorname{fin}}(t)\neq 0$ on the image $R_{\operatorname{fin}}$ of this subalgebra under the natural surjection $R^I\to R^I/\cU=\Ult_\cU(R,1)$.
By definition of $\cC$, the image of $\Delta_R$ is contained in $R_{\operatorname{fin}}$. Thus 
$$D := \partial_{\operatorname{fin}} \circ \Delta_R\colon R \to \Ult_\cU(R,1)$$ is a $k$-derivation and $D(t) \neq 0$. Then $$\partial := {^*}D\colon{^*}R \to {^*}\!\Ult_\cU(R,1) = {^*}R$$ is our desired derivation.
\end{proof}

\noindent
We specialize the above result to obtain our derivation which is nontrivial on the non-standard integers.  
Below we fix an arbitrary non-principal ultrafilter $\widetilde{\cU}$ (on some unspecified index set), and
given a ring $A$ we write $\widetilde{A}=A^{\widetilde{\cU}}$.

\begin{corollary}
\label{eederZ}
There is
a $k$-derivation $\partial$ on an ultralimit ${^*}\widetilde{R}$ of $\widetilde{R}$ such that $\partial(t)\neq 0$ for some  $t\in\widetilde{\ZZ}$.
\end{corollary}
\begin{proof} 
Let 
$t \in \widetilde\ZZ \setminus \ZZ$ be an arbitrary new element of $\widetilde\ZZ$.
Then $t$ is transcendental over $k$, and the previous lemma applies to $\widetilde{R}$ in place of $R$.
\end{proof}

\noindent
Combining the previous corollary with Lemma~\ref{annprime}, we conclude that noetherian rings having torsion-free nilpotent elements have elementary extensions with an automorphism moving the nonstandard integers.

\begin{lemma}
\label{lem:eeaut}
Let $A$ be a noetherian ring with nilradical $N=N(A)$, and suppose that $\ann_\ZZ(N) = 0$. Then there is an ultralimit ${^*}\widetilde{A}$ of $\widetilde{A}$ and an automorphism $\sigma$ of~${^*}\widetilde{A}$ over $A$ for which $\sigma({^*}\widetilde\ZZ) \nsubseteq {^*}\widetilde\ZZ$.
\end{lemma}
\begin{proof}
Let $\epsilon$ be an element of $N$ with $\ann_A(\epsilon) =: \fq$ prime and $\ann_\ZZ(\epsilon) = 0$, given by Lemma~\ref{annprime}.  Let $\pi\colon A \to A/\fq =: R$ be the natural quotient map. By Corollary~\ref{eederZ}, we can find a limit ultrapower ${^*}\widetilde{R}$ of $\widetilde{R}$ and an $R$-derivation $\partial\colon{^*}\widetilde{R} \to {^*}\widetilde{R}$ which is nontrivial on ${^*}\widetilde\ZZ$.   Note that $A\epsilon$ is an $R$-module in a natural way, and so ${^*}\widetilde{A}\epsilon$ is an ${^*}\widetilde{R}$-module.
We thus may define a map
$\sigma\colon{^*} \widetilde{A} \to {^*}\widetilde{A}$ by $x \mapsto x + \partial({^*}\pi(x)) \epsilon$; one checks easily that $\sigma$ is an automorphism over $A$.
We have $A\varepsilon\cap\ZZ=0$ and $\ann_R(\epsilon)=0$, 
hence ${^*}\widetilde{A}\varepsilon\cap {^*}\widetilde{\ZZ}=0$ and $\ann_{{^*}\widetilde{R}}(\epsilon) = 0$.
Since $\partial$ is nontrivial on ${^*}\widetilde{\ZZ}$, it  follows that $\sigma({^*}\widetilde{\ZZ}) \nsubseteq {^*}\widetilde{\ZZ}$.
\end{proof}

\noindent
We conclude that  rings as in the previous lemma are not bi-interpretable with~$\ZZ$.

\begin{cor}
\label{cor:nbinil}
No noetherian ring with   nilradical $N$ for which $\ann_\ZZ(N) = 0$ is bi-interpretable with $\ZZ$.
\end{cor}
\begin{proof}
Combine Lemma~\ref{lem:eeaut} and Corollary~\ref{cor:biNdefZ}.
\end{proof}

\noindent
We finish this subsection by remarking that although it may not be obvious from the outset, a non-trivial derivation on a proper elementary extension of $\tilde{R}$ as constructed in Corollary~\ref{eederZ} has some unexpected properties (not exploited in the present paper):

\begin{lemma}
\label{lem:divisiblederivations}
Let $Z\succeq \ZZ$ and $\partial\colon Z\to Z$ be a derivation. Then $\partial(Z)\subseteq\bigcap_{n\geq 1} nZ$.  \end{lemma}
\begin{proof}
Let $a\in Z$ and $n\geq 1$; we need to show that $\partial(a)$ is divisible by $n$, and for this we may assume that $a\geq 0$. By the Hilbert-Waring Theorem we may write $a=\sum_{i=1}^{g} b_i^n$ for some $b_i\in Z$ (where $g=g(n)$ only depends on $n$), and differentiating both sides of this equation yields
$\partial(a) = n \sum_{i=1}^{g} b_i^{n-1}\partial(b_i)$.
\end{proof}


\subsection{Finishing the proof of the main theorem} 
We now complete the proof of the main theorem stated in the introduction, along the lines of the argument
sketched there: Let $A$ be a f.g.~ring, $N=N(A)$.
Suppose $\ann_\ZZ(N)\neq 0$. Then by Theorem~\ref{thm:nilpotent}
(applied to $A$ in place of $B$), the rings $A$ and $A_{\operatorname{red}}$ are bi-interpretable,
and by Theorem~\ref{thm:biNreduced}, the reduced ring $A_{\operatorname{red}}$ is bi-interpretable with $\NN$ if and only if~$A_{\operatorname{red}}$ is infinite and
$\Spec^\circ(A_{\operatorname{red}})$ is connected.  The latter is equivalent to $A$ being infinite and $\Spec^\circ(A)$ being
connected.
If $\ann_\ZZ(N)=0$, then $A$ is not bi-interpretable with $\ZZ$, by Corollary~\ref{cor:nbinil}. \qed

\section{Quasi-Finite Axiomatizability}

\noindent
In this section we show the corollary stated in the introduction, in a slightly more precise form:

\begin{prop}\label{prop:QFA fg ring}
Every f.g.~ring has a QFA formula.
\end{prop}

\noindent
Throughout this section we let $A$, $B$ be f.g.~rings.

\begin{lemma}\label{lem:module QFA}
Suppose there is a QFA formula for $A$, and let $M$ be a f.g.~$A$-module. Then there is a QFA formula for the $A$-module $M$ viewed as a two-sorted structure.
\end{lemma}
\begin{proof}
Let $\varphi_A(x)$ be a QFA formula for $A$, where $x=(x_1,\dots,x_m)$. So we can take generators $a_1,\dots,a_m$ of $A$
such that for each f.g.~ring $A'$ and $a_1',\dots,a_m'\in A'$, we have $A'\models\varphi_A(a_1',\dots,a_m')$ iff there is an isomorphism $A\to A'$ with $a_i\mapsto a_i'$ for $i=1,\dots,m$. (Note that there is at most one such isomorphism $A\to A'$.)
Let also $b_1,\dots,b_n$ be generators for $M$. Since $A$ is noetherian, the syzygies of these generators are f.g., that is, there are  $a_{jk}$ ($j=1,\dots,n$, $k=1,\dots,p$)  of $A$ such that for all $\alpha_1,\dots,\alpha_n\in A$ we have
\begin{multline*}
\sum_{j=1}^n \alpha_j b_j=0 \quad\Longleftrightarrow\quad\\
\text{there are $\beta_1,\dots,\beta_p\in A$ such that $\alpha_j=\sum_{k=1}^p \beta_k a_{jk}$ for all $j=1,\dots,n$.}
\end{multline*}
For   $j=1,\dots,n$, $k=1,\dots,p$
pick polynomials $P_{jk}\in\ZZ[x]$ such that $a_{jk}=P_{jk}(a)$, where $a=(a_1,\dots,a_m)$.
Let 
$y=(y_1,\dots,y_n)$ be a tuple of distinct variables of the module sort, $u$ be another variable of the module sort, and
$u_1,\dots,u_p,z_1,\dots,z_n$ be distinct new variables of the ring sort. Let
$\gamma(y)$ be the formula
$$\forall u\exists z_1\cdots\exists z_n \left(u=\sum_{i=1}^n z_iy_i\right)$$
and $\zeta(x,y)$  be the formula
$$\forall z_1\cdots\forall z_n \left(\sum_{j=1}^n z_ny_n=0 
\ \longleftrightarrow \   
\exists u_1\cdots\exists u_p \left( \bigwedge_{j=1}^n z_j = \sum_{k=1}^p u_k P_{jk}(x)\right)\right).
$$
Finally, let $\alpha$ be a sentence expressing that $A$ is a ring and $M$ is an $A$-module.
One verifies easily that $$\varphi_M(x,y):=\alpha\wedge\varphi_A(x)\wedge \gamma(y)\wedge  \zeta(x,y)$$
is a QFA formula for the two-sorted structure $(A,M)$. 
\end{proof}

\begin{lemma}\label{lem:exist morphism}
Let $a_1,\dots,a_n$ be generators for $A$. There is a formula $\mu(x_1,\dots,x_n)$ such that
for all rings $A'$ and $a_1',\dots,a_n'\in A'$, we have: $A'\models\mu(x_1,\dots,x_n)$ iff
there is a morphism $A\to A'$ with $a_i\mapsto a_i'$ for $i=1,\dots,n$.
\end{lemma}
\begin{proof}
Let $x=(x_1,\dots,x_n)$ and  $\pi\colon\ZZ[x]\to A$ be the (surjective) ring morphism satisfying $\pi(x_i)=a_i$ for $i=1,\dots,n$.
Let $P_1,\dots,P_m\in\ZZ[x]$ generate the kernel of~$\pi$ and let
$\mu(x)$ be the formula $P_1(x)=\cdots=P_m(x)=0$.
\end{proof}

\begin{lemma}
Let $I$, $J$ be ideals of $B$ such that $IJ=0$. If there are QFA formulas for $B/I$ and for $B/J$, then there is one for $B$.
\end{lemma}
\begin{proof} 
Let $\varphi_I(x)$ be a QFA formula for $A=B/I$, where $x=(x_1,\dots,x_m)$, and take 
a system  $a=(a_1,\dots,a_m)$ of generators of $A$
such that for each f.g.~ring $A'$ and $a'=(a_1',\dots,a_m')\in (A')^m$, 
we have $A'\models\varphi_I(a')$ iff there is an isomorphism~$A\to A'$ with $a\mapsto a'$. 
Take generators $b_1,\dots,b_m,f_1,\dots,f_p$ for the ring $B$ such that $a_i=b_i+I$ for $i=1,\dots,m$
and $I=(f_1,\dots,f_p)$. Put $b=(b_1,\dots,b_m)$, $f=(f_1,\dots,f_p)$.
From our QFA formula $\varphi_I(x)$ for $A$ we easily construct a formula  $\psi_I(x,u)$  (where $u=(u_1,\dots,u_p)$) such that for each f.g.~ring $B'$ and $b'=(b_1',\dots,b_m')\in (B')^m$, $f'=(f_1',\dots,f_p')\in (B')^p$, the following are equivalent, with $I':=(f_1',\dots,f_p')\subseteq B'$:
\begin{itemize}
\item[(1)] $B'\models\psi_I(b',f')$;
\item[(2)] there is an isomorphism $A\to B'/I'$ with $a\mapsto b'+(I')^m$. 
\end{itemize}
(See Example~\ref{ex:ring interpretations},~(2).)
By the preceding lemma, there is also a QFA formula for the two-sorted structure~$(B/J,I)$.
Hence as before, we can  take generators $c_1,\dots,c_n,g_1,\dots,g_q$ of~$B$ such that the cosets
$c_1+J,\dots,c_n+J$ generate the ring~$B/J$ and $g_1,\dots,g_q$ generate the ideal~$J$, as well as a formula
$\psi_J(y,u,v)$, where $y=(y_1,\dots,y_n)$, $v=(v_1,\dots,v_q)$, such that
 for each f.g.~ring $B'$ and tuples $c'=(c_1',\dots,c_n')$, $f'=(f_1',\dots,f_p')$, 
and $g'=(g_1',\dots,g_q')$ of elements of $B'$, the following statements are equivalent, with $J':=(g_1',\dots,g_q')\subseteq B'$:
\begin{itemize}
\item[(3)] $B'\models\psi_J(c',f',g')$;
\item[(4)]  there is an isomorphism $(B/J,I)\to (B'/J',I')$ with 
$c+J^n\mapsto  c'+(J')^n$  and  $f\mapsto f'$.
\end{itemize}
Now by Lemma~\ref{lem:exist morphism} let $\mu(x,y,u,v)$ be a formula such that for each f.g.~ring $B'$ and 
tuples  $b'=(b_1',\dots,b_m')$, $c'=(c_1',\dots,c_n')$, $f'=(f_1',\dots,f_p')$,
$g'=(g_1',\dots,g_q')$ of elements of $B'$, the following are equivalent:
\begin{itemize}
\item[(5)] $B'\models\mu(b',c',f',g')$;
\item[(6)] there is a morphism $B\to B'$ with $b\mapsto b'$, $c\mapsto c'$, $f\mapsto f'$,
and $g\mapsto g'$.
\end{itemize}
Then by Lemma~\ref{lem:criterion for bijection} and the equivalences of (1), (3), (5) with (2), (4), (6), respectively,  the formula
$\psi_I(x,u)\wedge \psi_J(y,u,v) \wedge \mu(x,y,u,v)$
is QFA for $B$ with respect to the system of generators $b$,~$c$,~$f$,~$g$ of $B$.
\end{proof}

\begin{cor}
Let $N_1,\dots,N_e$ \textup{(}$e\geq 1$\textup{)} be ideals of $B$ such that $N_1\cdots N_e=0$.
Suppose that for $k=1,\dots,e$, there is a QFA formula for the f.g.~ring $B/N_k$. Then there is
a QFA formula for $B$.
\end{cor}
\begin{proof}
We proceed by induction on $e$. The case $e=1$ being trivial, suppose that $e\geq 2$ and put $I:=N_1\cdots N_{e-1}$, $J:=N_e$, so $IJ=0$. By assumption, there is a QFA formula for $B/J=B/N_e$.
Consider the f.g.~ring $\overline{B}:=B/I$ and the ideals $\overline{N_k}:=N_k+I$ ($k=1,\dots,e-1$)  of $\overline{B}$. We have $\overline{N_1}\cdots\overline{N_{e-1}}=0$, and 
the residue map $B\to\overline{B}$ induces an isomorphism $B/N_k\to \overline{B}/\overline{N_k}$. Hence by the inductive hypothesis applied to $\overline{B}$ and $\overline{N_1},\dots,\overline{N_{e-1}}$, there is a QFA formula for $\overline{B}=B/J$. Now by the proposition above,
there is a QFA formula for $B$.
\end{proof}

\noindent
We can now prove Proposition~\ref{prop:QFA fg ring}. First,
applying the previous corollary to $N_1=\cdots=N_e=N(B)$ where $e=$~nil\-po\-tency index of $N(B)$ yields that if there is a QFA formula for $B_{\operatorname{red}}$, then there is a QFA formula for $B$.
Thus to show that~$B$ has a QFA formula we may assume that~$B$ is reduced. Let $P_1,\dots,P_e$ be the minimal prime ideals of $B$. Then $P_1\cdots P_e=P_1\cap\cdots\cap P_e=0$, and by Corollary~\ref{cor:QFA int domains}, for each $k=1,\dots,e$ there is a QFA formula for the f.g.~integral domain $B/P_k$. Hence again by the preceding corollary, there is a QFA formula for $B$. \qed

\end{document}